\documentclass[11pt,reqno]{amsart}

\usepackage[utf8]{inputenc}
\usepackage[T1]{fontenc}
\usepackage[svgnames,hyperref]{xcolor}
\usepackage{lmodern,amssymb}
\usepackage[UKenglish,USenglish,english,american]{babel}
\usepackage[left=3.5cm,right=3.5cm,top=2.5cm,bottom=3cm]{geometry}

\usepackage[cmtip]{xy}
\xyoption{pdf}
\xyoption{color}
\xyoption{all}

\usepackage[shortlabels]{enumitem}
\setlist[enumerate]{label=\rm{(\arabic*)}}
\setlist[enumerate,2]{label=\rm({\it\roman*})}
\setlist[itemize]{label=\raisebox{0.25ex}{\tiny$\bullet$}}

\theoremstyle{plain}    
 \newtheorem{thm}{Theorem}[section]
 \numberwithin{equation}{section} 
 \numberwithin{figure}{section} 
 \theoremstyle{plain}
 \theoremstyle{plain}    
 \theoremstyle{plain}    
 \newtheorem{prop}[thm]{Proposition} 
 \theoremstyle{plain}    
 \newtheorem{lem}[thm]{Lemma} 
 \theoremstyle{remark}
 \newtheorem{rmk}[thm]{Remark}
 \theoremstyle{definition}

 \theoremstyle{plain}    
 \newtheorem{conj}[thm]{Conjecture} 
\theoremstyle{definition}
\newtheorem{defi}[thm]{Definition}

\newcommand{\C}{{\mathbb{C}}}
\newcommand{\N}{{\mathbb{N}}}

\newcommand{\R}{{\mathbb{R}}}

\newcommand{\e}{\varepsilon}

\newcommand{\f}{\varphi}

\DeclareMathOperator{\tr}{tr}

\newcommand{\mycolor}{Navy}
\usepackage[pdfauthor={Dat T. T\^O}, colorlinks, linktocpage, citecolor = \mycolor, linkcolor = \mycolor, urlcolor = \mycolor]{hyperref}
\usepackage[all]{hypcap} 

\title[REGULARIZING PROPERTIES OF COMPLEX MONGE-AMP\`ERE FLOWS II]{ Regularizing properties of Complex Monge-Amp\`ere flows II: Hermitian manifolds}
\author[T. D. T\^O]{Tat Dat T\^O}
\date{\today}
\address{ Institut Math\'ematiques de Toulouse \\
Universit{\'e} Paul Sabatier\\ 31062 Toulouse cedex 09\\ France.}

\email{tat-dat.to@math.univ-toulouse.fr}

\begin{document}
\begin{abstract}
We prove that a general complex Monge-Amp\`ere flow  on a Hermitian manifold can be run from an arbitrary initial condition  with zero Lelong number at all points. Using this property, we confirm  a conjecture of Tosatti-Weinkove: the Chern-Ricci flow performs a {\sl canonical surgical contraction}. Finally, we study a generalization of the Chern-Ricci flow on compact Hermitian manifolds, namely the twisted Chern-Ricci flow. 
\end{abstract}
\maketitle

\section*{Introduction}
 Let $(X,g,J)$ be a compact Hermitian  manifold of  complex dimension $n$, that is a compact complex manifold such that $J$ is compatible with the Riemannian metric $g$. 
Recently a number of geometric flows have been introduced to study the structure of Hermitian manifolds. Some flows which do preserve the Hermitian property have been proposed by Streets-Tian \cite{StT10,StT11,StT13}, Liu-Yang \cite{LY12} and also anomaly flows due to Phong-Picard-Zhang \cite{PPZ15,PPZ16a,PPZ16b} which moreover preserve the conformally balanced condition of Hermitian metrics. Another such flow, namely the Chern-Ricci flow, was introduced by Gill \cite{Gil11} and has been further  developed by Tosatti-Weinkove in \cite{TW15}. The Chern-Ricci flow is written as 
\begin{equation}\label{CRF}
\frac{\partial}{\partial t}\omega=-Ric(\omega),\quad \omega|_{t=0}=\omega_0,
\end{equation}
where $Ric(\omega)$ is the {\sl Chern-Ricci form} which is defined locally by
$$Ric(\omega):=-dd^c\log\omega^n:=-\frac{\sqrt{-1}}{\pi}\partial\bar{\partial}\log\omega^n.$$
This flow specializes the K\"ahler-Ricci flow when the initial metric is K\"ahler.  In \cite{TW15,TW13} Tosatti and Weinkove have investigated the flow on arbitrary Hermitian manifolds, notably in complex dimention 2 (see  also \cite{TW15b,FTW,GS15,Gil15,LV13,Zhe15,Yang16} for more recent works on the Chern-Ricci flow).  

\medskip 
For the K\"ahler case,  running the K\"ahler-Ricci flow (or complex Monge-Amp\`ere flows) from a rough initial data has been studied by several recent works \cite{CD07}, \cite{ST}, \cite{SzTo}, \cite{GZ13}, \cite{BG13}, \cite{DiNL14}. In \cite{ST}, \cite{SzTo} the authors succeeded to run certain complex Monge-Amp\`ere flows from continuous initial data, while \cite{DiNL14} and \cite{GZ13} are running a simplified flow starting from an initial current with zero Lelong numbers. Recently, we extended these latter works to deal with general complex Monge-Amp\`ere flows and arbitrary initial condition (cf. \cite{To16}). One of the motivations for this problem comes from the Analytic Minimal Model Program proposed by Song-Tian \cite{ST}.  For the Chern-Ricci flow, the same question was asked recently by Tosatti-Weinkove \cite{TW13,TW15} related to the classification of non-K\"ahler complex surfaces. 

\medskip
\medskip
Assume that there exists a holomorphic map between compact Hermitian manifolds $\pi:X\rightarrow Y$  blowing down an exceptional divisor $E$ on $X$ to one point $y_0\in Y$. In addition, assume that there exists a smooth function $\rho$ on $X$ such that
\begin{equation}\label{condition of cohomology class}
\omega_0-TRic(\omega_0)+dd^c\rho=\pi^*\omega_Y,
\end{equation}
with $T<+\infty$. Tosatti and Weinkove proved: 

\medskip
\noindent {\bf Theorem.}(\cite{TW15,TW13}){\it The solution $\omega_t$ to the Chern-Ricci flow (\ref{CRF}) converges in $C_{loc}^\infty(X\setminus E)$ to a smooth Hermitian metric $\omega_T$ on $X\setminus E$.}

\medskip
{\it Moreover, there exists a distance function $d_T$ on $Y$ such that $(Y,d_T)$ is a compact metric space and $(X,g(t))$ converges in the Gromov-Hausdorff sense $(Y,d_T)$ as $t\rightarrow T^-$.
} 
 
 \medskip
 \medskip
Observe that $\omega_T$ induces a singular metric $\omega'$ on $Y$ which is smooth in $Y\setminus\{y_0\}$. Tosatti and Weinkove conjectured that one can continue the Chern-Ricci flow on $Y$ with initial data $\omega'$. This is an open  question in \cite[Page 2120]{TW13} in which they conjectured that the Chern-Ricci flow performs a {\sl canonical surgical contraction}:
 
 \medskip
\noindent {\bf Conjecture.} (Tosatti-Weinkove \cite[Page 2120]{TW13}){\it 
\begin{enumerate}
\item There exists a smooth maximal solution $\omega_t$ of the Chern-Ricci on $Y$ for $t\in (T,T_Y)$ with $ T<T_Y\leq+\infty$ such that $\omega_t$ converges to $\omega'$, as $t\rightarrow T^+$, in $C^{\infty}_{loc}(Y\setminus\{y_0\})$. Furthermore, $\omega_t$ is uniquely determined by $\omega_0$. 

\medskip
\item The metric space $(Y,g(t))$ converges to $(Y,d_T)$ as $t\rightarrow T^+$ in the Gromov-Hausdorff sense. 
\end{enumerate}
}
 \medskip 
\medskip
In this note, we confirm this conjecture\footnote{After this paper was completed, the author learned that Xiaolan Nie proved the first statement of the conjecture for complex surfaces (cf. \cite{Nie2}). She also proved that the Chern-Ricci flow can be run from a bounded data. The author would like to thank Xiaolan Nie for sending her preprint.} .
An essential ingredient of its proof is to prove that the Monge-Amp\`ere flow corresponding to the Chern-Ricci flow can be run from a rough data.  By generalizing a result of Sz\'ekelyhidi-Tosatti \cite{SzTo}, Nie \cite{Nie} has  proved this property for compact Hermitian manifolds of vanishing first Bott-Chern class and continous initial data. In this paper,  we generalize the previous results of Nie \cite{Nie} and the author  \cite{To16} by  considering the following complex Monge-Amp\`ere flow:
\begin{equation*}
(CMAF) \hskip1cm 
\dfrac{\partial \f_t}{\partial t}= \log\dfrac{(\theta_t+dd^c\f_t)^n}{\Omega}-F(t,x,\f_t),
\end{equation*} 
where $(\theta_t)_{t\in [0,T]}$ is a family of Hermitian forms with $\theta_0=\omega$ and $F$ is a smooth function on $\R\times X\times \R$.

\medskip
\medskip
\noindent\textbf{Theorem A.}
\textit{Let $\f_0$ be a $\omega$-psh function  with zero Lelong number at all points. Let $(t,z,s)\mapsto F(t,z,s)$ be a smooth function on $[0,T]\times X\times \R$ such that $\partial F/\partial s\geq 0$ and $\partial F/\partial t$ is bounded from below.}

\medskip
{\it Then there exists a family of smooth strictly $\theta_t-psh$ functions $(\varphi_t)$ satisfying $(CMAF)$
in $(0, T]\times X,$ with $\varphi_t\rightarrow \varphi_0$ in $L^1(X),$ as $t\searrow 0^+$ and $\f_t$ converges to $\f_0$ in $C^0(X)$ if $\f_0$ is continuous. This family is moreover unique if $\partial F/\partial t$ is bounded and $\partial F/\partial s\geq 0$.}

\medskip
The following stability result is a straighforward extension of \cite[Theorem 4.3, 4.4]{To16}.

\medskip
\medskip
\noindent \textbf{Theorem B.} \textit{Let $\varphi_0,\varphi_{0,j}$ be $\omega$-psh functions with zero Lelong number at all points, such that $\f_{0,j}\rightarrow \f_0$ in $L^1(X)$. Denote by $\varphi_{t,j}$ and $\f_t$ the corresponding solutions of $(CMAF)$ with initial condition $\f_{0,j}$ and $\f_0$ respectively. Then for each $\e\in(0,T)$
$$\f_{t,j}\rightarrow \f_{t}\  \text{ in }\  C^\infty ([\e, T]\times X)\ \text{ as }\  j\rightarrow +\infty.$$
\quad Moreover, if $\f_0$ and $\psi_0$ are continuous, then for any $k\geq 0$, for any $0<\e<T$, there exists a positive constant $C(k,\e)$ depending only on $k$ and $\e$ such that
\begin{equation*}
||\f-\psi||_{C^{k}([\e,T]\times X)}\leq C(k,\e)||\f_0-\psi_0||_{L^\infty(X,\omega)
}.
\end{equation*}}

As a consequence of Theorem A and Theorem B, the Chern-Ricci flow on any Hermitian manifold can be run from rough data.  Using this result and a method due to Song-Tosatti-Weinkove \cite{SW13,TW13} we prove the conjecture. The proof  is given in Section \ref{proof of conjecture}.

\medskip
The second purpose of this paper is to study a generalization of the Chern-Ricci flow, namely the {\sl twisted Chern-Ricci flow}: 
 \begin{equation*}
 \frac{\partial\omega_t }{\partial t}= -Ric(\omega_t)+\eta, \quad \omega|_{t=0}=\omega_0
 \end{equation*}
 where $Ric(\omega_t)$ is the Chern-Ricci form of $\omega_t$,  $\omega_0$ is a Hermitian metric on $X$ and $\eta$ is a smooth $(1,1)$-form. In general, we do not assume $\eta$ is closed. This flow also generalizes the twisted K\"ahler-Ricci flow which has been studied recently  by several authors (see for instance \cite{CSz12,GZ13}).

\medskip 
We show that the twisted Chern-Ricci flow starting from a Hermitian metric $\omega_0$  is equivalent to the following complex Monge-Amp\`ere flow 
\begin{equation}
\frac{\partial \f}{\partial t}=\log\frac{(\hat{\omega}_t+dd^c \f)^n}{\omega_0^n},
\end{equation}
where $\hat{\omega}_t=\omega_0+t(\eta-Ric(\omega_0))$. We first prove the following, generalizing \cite[Theorem 1.2]{TW15}:

\medskip
\medskip
\noindent {\bf Theorem C.} {\it There exists a unique solution to the twisted Chern-Ricci flow on $[0,T)$, where 
$$T:=\sup\{t\geq 0| \exists \psi\in C^\infty(X) \text{ with } \hat{\omega}_t+dd^c\psi>0\}.$$ }

\medskip
When the twisted Chern-Ricci flow has a long time solution, it is natural to study its behavior at infinity. When the Bott-Chern class vanishes and $\eta=0$,  Gill  has proved that the flow converges to a Chern-Ricci flat Hermitian metric (cf. \cite{Gil11}). 

\medskip
Denote by
\begin{equation*}
\{\eta\}:=\{ \alpha \text{ is  a real (1,1)-form }| \exists f\in C^\infty(X) \text{ with }   \alpha= \eta+dd^c f \},
\end{equation*} 
the  equivalence class of $\eta$. Suppose that $c^{BC}_1(X)-\{\eta\}$ is negative.  Consider the {\sl normalized twisted Chern-Ricci flow}
\begin{equation}\label{NTCRF}
\frac{\partial\omega_t }{\partial t}= -Ric(\omega_t)-\omega_t+\eta.
\end{equation}
Then we have the following result for the long time behavior of the flow generalizing \cite[Theorem 1.7]{TW15}:

\medskip
\medskip
\noindent\textbf{Theorem D.} {\it  Suppose $c^{BC}_1(X)-\{\eta\}< 0$.  The normalized twisted  Chern-Ricci flow  smoothly converges to a Hermitian metric $\omega_{\infty}=\eta-Ric(\Omega)+dd^c \f_{\infty}$ which satisfies
  $$Ric(\omega_\infty)=\eta-\omega_\infty.$$}

\medskip
Observe that $\omega_\infty$ satisfies the {\sl twisted Einstein equation}:
\begin{equation}\label{teq}
Ric(\omega)=\eta-\omega.
\end{equation}
We can prove the existence of a unique solution of (\ref{teq}) using a result of Monge-Amp\`ere equation due to Cherrier \cite{Cher87} (see Theorem \ref{twisted Einstein metric}).
Theorem D moreover gives  an alternative proof of the existence of the twisted Einstein metric $\omega_\infty$ in $-c^{BC}_1(X)+\eta$. This is therefore a generalization of Cao's approach \cite{Cao85} by using K\"ahler-Ricci flow to prove the existence of K\"ahler-Einstein metric on K\"ahler manifold of negative first Chern class.
In particular, when $\eta=0$, we have $c_1^{BC}(X)<0$ hence we have $c_1(X)<0$ and $X$ is a K\"ahler manifold, this is \cite[Theorem 1.7]{TW15}. 

\medskip
Note that in general, one cannot assume $\eta$ to be closed, in contrast with the twisted K\"ahler-Ricci flow. Let us stress also that the limit of the normalized twisted Chern-Ricci flow exists without assuming that the manifold is K\"ahler (a necessary assumption when studying the long term behavior of the Chern-Ricci flow). Therefore  the twisted Chern-Ricci flow is somehow more natural in this context.

\medskip
As an application of Theorem D, we give an alternative proof of the existence of a unique smooth solution for the following  Monge-Amp\`ere equation 
$$(\omega+dd^c\f)^n=e^{\f}\Omega.$$
We show that the solution is the limit of the potentials of a suitable twisted normalized  Chern-Ricci flow. Cherrier \cite {Cher87} proved this result by generalizing the elliptic approach of Aubin \cite{Aub78} and \cite{Yau}.  

\medskip
The paper is organized as follows. In Section \ref{strategy}, we recall some notations in Hermitian manifolds. In Section \ref{a priori estm} we prove  various {\sl a priori} estimates following our previous work \cite{To16}. The main difference is that we will use the recent result of Ko\l{}oziedj's uniform type estimates for Monge-Amp\`ere on Hermitian manifolds (cf. \cite{DK12, Blo11, Ng16}) instead of the one on K\"ahler manifolds to bound the oscillation of the solution.  The second arises when estimating the gradient and the Laplacian: we use a special local coordinate system due to Guan-Li \cite[Lemma 2.1]{GL10} instead of the usual normal coordinates in K\"ahler geometry. In Section \ref{proof} we prove Theorem B and Theorem C. In Section \ref{proof of conjecture}, we prove the conjecture.  In Section \ref{definition of CRF} we define the twisted Chern-Ricci flow and prove the existence of a unique maximal solution using the estimates in Section \ref{a priori estm}. The approach is different from the one for the Chern-Ricci flow due to Tosatti-Weinkove \cite{TW15}. We also show that the twisted Chern-Ricci flow on negative twisted Bott-Chern class smoothly converges to the unique twisted Einstein metric. 

\medskip
\textbf{Acknowledgement.} The author is grateful to his supervisor Vincent Guedj for support, suggestions and encouragement.  The author  thanks Valentino Tosatti and Ben Weinkove for their interest in this work and helpful comments. We also thank Thu Hang Nguyen, Van Hoang Nguyen and Ahmed Zeriahi  for very useful discussions. The author would like to thank  the referee for useful comments and suggestions. This work is supported by the Jean-Pierre Aguilar fellowship of the CFM foundation.
\section{Preliminaries}\label{strategy}
\subsection{Chern-Ricci curvature on Hermitian manifold}
Let $(X,g)$ be a compact Hermitian manifold of complex dimension $n$. In local coordinates, $g$ is determined by the $n\times n$ Hermitian matrix $(g_{i\bar{j}})=g(\partial_i,\partial_{\bar{j}})$. We write $\omega=\sqrt{-1}g_{i\bar{j}}dz_i\wedge d\bar{z}_j$ for its associated $(1,1)$-form. 

\medskip
We define the Chern connection $\nabla$ associated to $g$ as follows.  If $X=X^j\partial_j$ is a vector field  and $\alpha=a_idz_i$ is a $(1,0)$-form then theirs covariant derivatives have components $$\nabla_i X^k=\partial_iX^j+ \Gamma^k_{ij}X^j,\quad \nabla_ia_j=\partial_i a_j -\Gamma_{ij}^ka_k,$$
where the Christoffel symbols $\Gamma^k_{ij}$ are given by
$$\Gamma^k_{ij}=g^{\bar{l}k}\partial_ig_{j\bar{l}}.$$ 

\medskip
We define the torsion tensors $T$ and $\bar{T}$ of $\omega$ as follows
\begin{align*}
T &=\sqrt{-1}\partial \omega =\frac{1}{2} T_{ij\bar{k}}dz_i\wedge dz_j\wedge d\bar{z}_k    \\
\bar{T}&=\sqrt{-1}\bar{\partial}\omega=\frac{1}{2}\bar{T}_{\bar{i}\bar{j}k}d\bar{z}_i\wedge d\bar{z}_j\wedge dz_k.
\end{align*}
where $$T_{ij\bar{k}}=\partial_{i}g_{j\bar{k}}-\partial_jg_{i\bar{k}}, \text{ and } \bar{T}_{\bar{i}\bar{j}k}=\partial_{\bar{j}}g_{k\bar{i}}-\partial_{\bar{i}}g_{k\bar{j}}.$$
Then the {\sl torsion tensor} of $\omega$ has component $$T^k_{ij}=\Gamma^k_{ij}-\Gamma^k_{ji}=g^{\bar{l}k}T_{ij\bar{l}}.$$

\begin{defi}
The {\sl Chern-Ricci curvature} of $g$ is  the tensor $$R_{k\bar{l}} (g):=R_{k\bar{l}} (\omega):=g^{\bar{j}i}R_{k\bar{l}i\bar{j}}=-\partial_k\partial_{\bar{l}}\log\det g,$$
and the {\sl Chern-Ricci form} is 
$$Ric(g):=Ric(\omega):=\frac{\sqrt{-1}}{\pi}R^C_{k\bar{l}}dz_k\wedge d\bar{z}_l=-dd^c \log\det g,$$
where  $$d:=\partial +\bar{\partial},\quad d^c:= \frac{1}{2i\pi}(\partial-\bar{\partial}).$$
\end{defi}
It is a closed real $(1,1)$-form and its cohomology class in the {\sl Bott-Chern cohomology group}
$$H^{1,1}_{BC}(X,\R):=\frac{\{\text{closed real (1,1)-forms}\}}{\{\sqrt{-1}\partial\bar{\partial}\psi, \psi\in C^\infty(X,\R)\}}$$
is the {\sl first Bott-Chern class} of $X$, denoted by  $c^{BC}_1(X)$, which is independent of the choice of Hermitian metric $g$. We also write $R=g^{k\bar{l}}R_{k\bar{l}}$ for the {\sl Chern scalar curvature}.

\subsection{Plurisubharmonic functions and Lelong number}
Let $(X,\omega)$ be a compact Hermitian manifold. 
\begin{defi}
We let $PSH(X,\omega)$ denote the set of all {\sl$\omega$-plurisubharmonic functions} ($\omega$-psh for short), i.e the set of functions $\f\in L^1(X,\R\cup \{-\infty\})$ which can be locally written as the sum of a smooth and a plurisubharmonic function, and such that$$ \omega+dd^c\f\geq 0$$
in the weak sense of positive currents. 
\end{defi}

\begin{defi}
Let $\f$ be a $\omega$-psh function and  $x\in X$. The {\sl Lelong number} of $\f$ at $x$ is
$$\nu(\f,x):=\liminf_{z\rightarrow x} \frac{\f(z)}{\log |z-x|}.$$
We say $\f$ has a {\sl logarithmic pole of coefficient $\gamma$} at $x$ if $\nu(\f,x)=\gamma$. 
\end{defi}

\section{ A priori estimates for complex Monge-Amp\`ere flows}\label{a priori estm}
In this section we prove various {\sl a priori} estimates for $\f_t$ which satisfies

\begin{equation*}
\frac{\partial \f_t}{\partial t}=\log  \frac{(\theta_t+dd^c \f_t)^n}{\Omega} -F(t,z,\f) \qquad \qquad (CMAF)
\end{equation*}
with a smooth strictly $\omega$-psh initial data $\varphi_0$, where $\Omega$ is a smooth volume form, $(\theta_t)_{t\in [0,T]}$ is a family of Hermitian forms on $X$ and $(t,z,s)\mapsto F(t,z,s)$ is a smooth function on $[0,T]\times X\times \R$ with  
\begin{equation}\label{condition of F}
\frac{\partial F}{\partial s}\geq 0 \text{ and } \frac{\partial F}{\partial t}> B,
\end{equation}
for some $B\in \R$.

\medskip
Since we are interested in the behavior near 0 of $(CMAF)$, we can further assume that
\begin{equation}
\frac{\omega}{2}\leq \theta_t\leq 2\omega\; \text{ and }\; \delta^{-1}\Omega \leq \theta^n_t\leq\delta \Omega,\forall t\in [0,T]\; \text{ for some } \delta>0, 
\end{equation}
\begin{equation}\label{condition of theta}
\theta_t-t\dot{\theta}_t\geq 0\, \text{ for }\, 0\leq t\leq T.
\end{equation}
The assumption (\ref{condition of theta}) will be used to bound $\dot{\f}_t$ from above. 
\subsection{Bounds on $\f_t$ and $\dot{\f}_t$}\label{C^0 estimates}
As in the K\"ahler case, the upper bound of $\f$ is a simple consequence of the maximal principle (see \cite[Lemma 2.1]{To16}). 

\medskip
 For a lower bound of $\f_t$, we have
\begin{lem}
There is a constant $C>0$  depending only on $\inf_X\f_0$ such that,
$$\f_t\geq \inf_X\f_0-Ct, \quad\forall (t,x)\in [0,T]\times X.$$
\end{lem}
\begin{proof}
Set $$\psi:=\inf_X\f_0-Ct,$$
where $C$ will be chosen hereafter.
Since we assume that $2\theta_t\geq \omega$,
$$\theta_t+dd^c\psi\geq  \frac{1}{2}\omega.$$
Combine with $\omega^n\geq 2^{-n}\theta^n_t\geq \Omega/(2^n\delta) $, we have
$$\frac{(\theta_t+dd^c\psi)^n}{\Omega}\geq\frac{1}{2^n}\frac{\omega^n}{\Omega}\geq \frac{1}{4^n\delta}.$$
We now choose $C>0$ satisfying 
$$-C+\sup_{[0,T]\times X}F(t,x,\inf_X\f_0)\leq \frac{1}{4^n\delta}.$$
hence 
$$\frac{\partial \psi_t}{\partial t}\leq \frac{(\theta_t+dd^c\psi)^n}{\Omega}-F(t,x,\psi),$$ 
 It follows from the maximum principle \cite[Proposition 1.5]{To16} that
 $$\f_t\geq \psi_t,$$
 as required.
\end{proof}
\medskip
For another lower bound, we follow the argument in \cite{GZ13}, replacing the uniform a priori bound of Ko\l{}odziej \cite{Kol98} by its Hermitian version (see for instance \cite[Theorem 2.1]{Ng16}).    First, we assume that $\theta_t\geq \omega+t\chi,\,\forall t\in [0,T],$ for some smooth $(1,1)$-form $\chi$. Let $0<\beta < +\infty$ be  such that 
$$\chi+(2\beta-1)\omega\geq 0.$$
It follows from  Ko\l{}odziej's uniform type estimate for Monge-Amp\`ere equation on Hermitian manifolds (cf. \cite[Theorem 2.1]{Ng16}) that the exists a continuous $\omega$-psh solution $u$ of the equation
\begin{equation*}
(\omega+dd^c u)^n=e^{ u-2\beta\f_0}\omega^n,
\end{equation*}  
which satisfies 
$$||u||_{L^\infty(X)}<C,$$
where $C$ only depends on $||e^{-2\beta\f_0}||_{L^p(X)}$, for some $p>1$.

\begin{rmk}
Latter on we will replace $\f_0$ by smooth approximants $\f_{0,j}$ of initial data. Since the latter one has zero Lelong numbers, Skoda's integrability theorem \cite{Sko} will provide a uniform bound for $||e^{-2\beta\f_0}||_{L^p(X)}$ and $||u||_{L^\infty(X)}$.
\end{rmk}
\begin{lem}\label{bound from below}
For all $z\in X $ and $0<t<\min(T,(2\beta)^{-1})$, we have
\begin{equation}
\f_t(z)\geq (1-2\beta t)\f_0(z)+t u(z)+n(t\log t-t)-At,
\end{equation}
where  $A$ depends on $\sup_X\f_0$. In particular, there exists $c(t)\geq 0$ such that
$$\f_t(z)\geq \f_0(z)-c(t),$$
 with $c(t)\searrow 0$ as $t\searrow 0$. 

\end{lem}
\begin{proof}
Set $$\phi_t:=(1-2\beta t)\f_0+ t u+n(t\log t-t)-At,$$  where $A:=\sup_{[0,T]\times X}F(t,z,C_0)$ with $(1-2\beta t)\f_0+ t u+n(t\log t-t)\leq C_0$ for all $t\in [0,\min(T,(2\beta)^{-1})]$. 

\medskip
By our choice of $\beta$ we have
\begin{eqnarray*}
\theta_t+dd^c\phi_t&\geq& \omega +t\chi+dd^c\phi_t \\
&=&(1-2\beta t)(\omega+dd^c\f_0)+ t (\omega +dd^c u)+ t[\chi+(2\beta-1 )\omega] \\
&\geq&t (\omega +dd^c u)\geq  0.
\end{eqnarray*}
Moreover 
$$ (\theta_t+dd^c \phi_t)^n\geq t^n(\omega+dd^c u)^n=e^{ \partial_t\phi_t +A}\geq e^{\partial_t \phi_t+F(t,z,\phi_t)},$$
hence $\phi_t$ is a subsolution to $(CMAF)$. Since $\phi_0=\f_0$ the conclusion follows from the maximum principle \cite[Proposition 1.5]{To16}.
\end{proof}

\medskip

The lower bound for $\dot{\f}$ comes from the same argument in \cite[Proposition 2.6]{To16}:

\begin{prop}\label{bound f'}
Assume $\f_0$ is bounded. There exist constants $A>0$ and $C=C(A, Osc_X\f_0)>0$ such that  for all $(x,t)\in X\times (0,T]$, 
$$\dot{\f}\geq n\log t-AOsc_X \f_0-C.$$
\end{prop}

\medskip
We now prove a crucial estimate for $\dot{\f}_t$ which allows us to use the uniform version of Kolodziej's uniform type estimates in order to get the bound of $Osc_X \f_t$. The proof is the same in \cite{GZ13,To16}, but we include a proof for  the reader's convenience.
\begin{prop}\label{bound f' above}
 There exists $0<C=C(\sup_X \varphi_0,T)$ such that for all $0< t\leq T$ and $z\in X$,
$$\dot{\f}_t(z)\leq \frac{-\f_0(z)+C}{t}.$$
\end{prop}
\begin{proof}
We consider $G(t,z)= t\dot{\f_t}-\f_t-nt + Bt^2/2 $, with $B$ is the constant in (\ref{condition of F}). We obtain
$$\left(\frac{\partial }{\partial t}-\Delta_{\omega_t} \right)G=-t\dot{\f}\frac{\partial F}{\partial s}+t\left( B-\frac{\partial F}{\partial t}\right) -\tr_{\omega_t}(\theta_t-t\dot{\theta}_t).$$
Since we assume that $\theta_t-t\dot{\theta}_t\geq 0$ (see (\ref{condition of theta})),  we get
$$\left(\frac{\partial }{\partial t}-\Delta_{\omega_t} \right)G\leq -t\dot{\f}\frac{\partial F}{\partial s}+t\left( B-\frac{\partial F}{\partial t}\right).$$
If $G$ attains its maximum at $t=0$, we have the result. Otherwise, assume that $G$ attains its maximum at $(t_0,z_0)$ with $t_0>0$, then at $(t_0,z_0)$ we have $$0\leq\left(\frac{\partial }{\partial t}-\Delta_{\omega_t} \right)G< -t_0\frac{\partial F}{\partial s}\dot{\f}.$$
Since $\frac{\partial F}{\partial s}\geq 0$ by the hypothesis, we obtain $\dot{\f}(t_0,z_0)<0$ and 
$$t\dot{\f_t}-\f_t-nt+Bt^2/2\leq -\f_{t_0}(z_0)-nt_0+Bt_0^2/2.$$ Using Lemma \ref{bound from below} we get $\f_{t_0}\geq \f_0-C_1$, where $C_1$ only depends on $\sup_X\f_0$ and $T$, hence there is a constant $C_2$ depending on $\sup_X\f_0$ and $T$ such that
$$t\dot{\f_t}\leq \f_t-\f_0+C_2.$$
Since $\f_t\leq C_3(\sup \f_0,T)$,
so $$\dot{\f_t}(x)\leq \dfrac{-\f_0+C}{t},$$ 
where $C$ only depends on $\sup_X\f_0$ and $T$.
\end{proof} 
\subsection{Bounding the oscillation of $\f_t$}
Once we get an upper bound for $\dot{\f_t}$ as in Proposition \ref{bound f' above}, we can bound the oscillation of $\f_t$ by using the following uniform version of Kolodziej's estimates due to Dinew- Ko\l{}odziej \cite[Theorem 5.2]{DK12}. 
\begin{thm}\label{DK} Let $(X,\omega)$ be a compact Hermitian manifold. Assume $\f\in C^2(X)$ is such that $\omega+dd^c\f\geq 0$ and 
$$(\omega+dd^c\f)^n=f\omega^n.$$
Then for $p>1$, 
$$Osc_X\f\leq C,$$
where $C$ only depends on $\omega, p,||f||_{L^p(X)}$. 
\end{thm}

Indeed, observe that $\f_t$ satisfies 
$$(\theta_t+dd^c\f_t)^n=H_t\Omega,$$
then by Proposition \ref{bound f' above}, for any $\e\in (0,T)$,
$$H_t=\exp(\dot{\f_t}+F)\leq \exp(\frac{-\f_0+C}{t}+C')$$
 for all $t\in [\e,T]$. 
Fix $p>1$ and $\mathcal F$ a compact family of $\omega$-psh functions with zero Lelong numbers, and assume that $\f_0\in \mathcal F$. It follows from the uniform version of Skoda's integrability theorem (cf. \cite[Proposition 7.1]{Sko} and \cite[Theorem 3.1]{Zer01}) that there exists $C_\e>0$ such that 
$$||e^{-\psi/t} ||_{L^p(\Omega)}\leq C_\e,$$ 
 for all $\psi\in \mathcal F, t\in [\e,T]$. We thus write for short $||H_t||_{L^p(\Omega)}\leq C(t)$ for some $C(t)>0$.
 
 \begin{rmk}
 Later on we will replace $\f_0$ by smooth approximants $\f_{0,j}$ of initial data. We can thus apply the previous estimate with $\mathcal{F}=\{\f_0\}\cup \{\f_{0,j}, j\in \N\}$, where $\f_0$  is now the initial data. This yields 
 $$||H_{t,j} ||_{L^{p}(\Omega)} \leq C(t).$$
 \end{rmk}

\medskip 
 Now, thanks to Theorem \ref{DK}, we infer that the oscillation of $\f_t$ is uniformly bounded:
\begin{thm}\label{bound oscillation}
Fix $0<t\leq T$. There exist $C(t)>0$ independent of $\inf_X\f_0$ such that
$$Osc_X(\f_t)\leq C(t).$$
\end{thm}
\subsection{Bounding the gradient of $\f$}
In this section we bound the gradient of $\f$ using the same technique as in \cite{To16} (see also \cite{SzTo}) which is a parabolic version of B\l{}ocki's estimate \cite{Blo09} for K\"ahler manifolds.  In these articles  we used the usual normal coordinates in K\"ahler geometry. For Hermitian manifolds, we need to use the following local coordinate system due to  Guan-Li \cite[Lemma 2.1]{GL10} (see also \cite{StT11} for a similar argument), which is also essential for our second order estimate. We also refer the reader to \cite[Lemma 6]{Ha96} for a gradient estimate for the elliptic Complex Mong-Amp\`ere equation in the Hermitian case without using the local coordinate system. We thank Valentino Tosatti for indicating the reference \cite{Ha96}. We remark that similar arguments of the proof below can be found in \cite[Lemma 3.3]{Nie}.

\medskip
\begin{lem}\label{GL coordinates}
At any point $x\in X$ there exists a local holomorphic coordinate system centered at $x$ such that for all $i,j$
\begin{equation}\label{GL local coor}
g_{i\bar{j}}(0)=\delta_{i,j}, \quad \frac{\partial g_{i\bar{i}}}{\partial z_j}(0)=0.
\end{equation}
\end{lem}

\medskip
We now prove
\begin{prop}\label{bound grad}
Fix $\e\in [0,T]$. There exists $C>0$ depending on $\sup_X \f_0$ and $\e$ such that for all $\e< t\leq T$
$$|\nabla \f(z)|^2_\omega< e^{C/(t-\e)}.$$
\end{prop}
\begin{proof}
Since the bound on $Osc_X\f_\e$ only depends on $\sup_X\f_0$ and $\e$ (see Theorem \ref{bound oscillation}),  we can consider the flow starting from $\f_\e$, i.e $\f(0,x)=\f_\e$. Then we need to show that there exists a constant $C$ depending on $Osc_X\f_0$ and $\e$ such that 
$$|\nabla \f(z)|^2_\omega< e^{C/t},$$
for all $t\in [0,T-\e]$.

\medskip
Define 
$$ K(t,x)=t\log|\nabla \f|^2_\omega -\gamma\circ \f=t\log \beta -\gamma\circ \f,$$
 for $(t,x)\in [0,T-\e]\times X$ where, $\beta=|\nabla\f|^2_\omega$ and $\gamma\in C^\infty(\R,\R)$ will be chosen hereafter.

\medskip
If $K(t,z)$ attains its maximum for $t=0$, $\beta$ is bounded in terms of $\sup_X\f_0$ and $\e$, since $|\f_t|$ is bounded by a constant depending on $\sup_X\f_0$ and $\e$ for all $t\in[0,T-\e]$ (see Section \ref{C^0 estimates}).
\medskip

We now assume that $K(t,z)$ attains its maximum at $(t_0,z_0)$ in $[0,T-\e]\times X$ with $t_0>0$. Near $z_0$ we have $\omega=\sqrt{-1} g_{i\bar{j}}dz_i\wedge d\bar{z}_j$ for some and  $\theta_t=\sqrt{-1} h_{i\bar{j}}dz_i\wedge d\bar{z}_j$. We take the local coordinates (\ref{GL local coor}) for $\omega$ at $z_0$ such that 
\begin{align}\label{1}
&g_{i\bar{k}}(z_0)=\delta_{jk}\\  \label{2}
&g_{i\bar{i}l}(z_0)=0\\ \label{3}
&u_{p\bar{q}}(t_0,z_0)=h_{p\bar{q}}+\f_{p\bar{q}}\ \text{is diagonal},
\end{align}
here for convenience we denote in local coordinate,  $u_p:=\frac{\partial  u}{\partial z_p}, u_{j\bar{k}}:=\frac{\partial^2 u}{\partial z_j\partial \bar{z}_k}$,  and $g_{i\bar{j}k}:=\frac{\partial  g_{i\bar{j}}}{\partial z_k}$.

\medskip We now compute $K_p,K_{p\bar{p}}$ at $(t_0,z_0)$ in order to use the maximum principle.  At $(t_0,z_0)$ we have $K_p=0$ hence
\begin{equation}\label{4}
t\beta_p=\beta \gamma'\circ \f\f_p
\end{equation}  
or  
$$\left(\frac{\beta_p}{\beta}\right)^2=\frac{1}{t^2}(\gamma'\circ \f)^2|\f_p|^2.  $$
Therefore,
\begin{align*}
K_{p\bar{p}}&= t\frac{\beta_{p\bar{p}}\beta-|\beta_p|^2}{\beta^2}-\gamma''\circ \f |\f_p|^2-\gamma'\circ \f\f_{p\bar{p}}\\
&=t\frac{\beta_{p\bar{p}}}{\beta}-[t^{-1}(\gamma'\circ \f)^2+\gamma''\circ \f]|\f_p|^2-\gamma'\circ \f\f_{p\bar{p}}.
\end{align*}

\medskip
Now we compute $\beta_p, \beta_{p\bar{p}}$ at $(t_0,z_0)$ with $\beta=g^{j\bar{k}}\f_j\f_{\bar{k}}$ where $(g^{j\bar{k}})=[(g_{j\bar{k}})^t]^{-1}$. We have
$$\beta_p=g^{j\bar{k}}_p\f_j\f_{\bar{k}}+g^{j\bar{k}}\f_{jp}\f_{\bar{k}}+g^{j\bar{k}}\f_{j}\f_{\bar{k}p}.$$
Since
$$g^{j\bar{k}}_p=-g^{j\bar{l}}g_{s\bar{l}p}g^{s\bar{k}},$$

$$\beta_p=- g^{j\bar{l}}g_{s\bar{l}p}g^{s\bar{k}} \f_j\f_{\bar{k}}+ g^{j\bar{k}}\f_{jp}\f_{\bar{k}}+g^{j\bar{k}}\f_{j}\f_{\bar{k}p}$$
and 
\begin{eqnarray*}
\beta_{p\bar{p}}&=&- g_{\bar{p}}^{j\bar{l}}g_{s\bar{l}p}g^{s\bar{k}} \f_j\f_{\bar{k}}- g^{j\bar{l}}g_{s\bar{l}p\bar{p}}g^{s\bar{k}} \f_j\f_{\bar{k}}- g^{j\bar{l}}g_{s\bar{l}p}g_{\bar{p}}^{s\bar{k}} \f_j\f_{\bar{k}}\\
&&- g^{j\bar{l}}g_{s\bar{l}p}g^{s\bar{k}} \f_{j\bar{p}}\f_{\bar{k}} - g^{j\bar{l}}g_{s\bar{l}p}g^{s\bar{k}} \f_j\f_{\bar{k}\bar{p}}  +g^{j\bar{k}}_{\bar{p}}\f_{jp}\f_{\bar{k}} +g^{j\bar{k}}\f_{jp\bar{p}}\f_{\bar{k}}\\
&&+g^{j\bar{k}}\f_{jp}\f_{\bar{k}\bar{p}}+g^{j\bar{k}}_{\bar{p}}\f_{j}\f_{\bar{k}p}+g^{j\bar{k}}\f_{j\bar{p}}\f_{\bar{k}p}+g^{j\bar{k}}\f_{j}\f_{\bar{k}p\bar{p}}.
\end{eqnarray*}
Therefore, at $(t_0,x_0)$,  $$g^{j\bar{k}}_p=-g_{k\bar{j}p},$$
\begin{equation}\label{bp}
\beta_p=- g^{j\bar{l}}g_{s\bar{l}p}g^{s\bar{k}} \f_j\f_{\bar{k}}+ \sum_j\f_{jp}\f_{\bar{j}}+\sum_j \f_{\bar{j}p}\f_j,
\end{equation}
and
\begin{eqnarray*}
\beta_{p\bar{p}}= -g_{k\bar{j}p\bar{p}}\f_j\f_{\bar{k}}+2Re(\f_{jp\bar{p}}\f_{\bar{j}})+ |\f_{jp}-\sum_k g_{j\bar{k}p}\f_k  |^2+| \f_{j\bar{p}} -\sum_kg_{\bar{k} j\bar{p}}\f_k|^2
\end{eqnarray*}
Now at $(t_0,z_0)$
$$
\Delta_{\omega_{t_{0}}}K=\sum_{p=1}^n \frac{K_{p\bar{p}}}{u_{p\bar{p}}}=\frac{t}{\beta}\frac{\beta_{p\bar{p}}}{u_{p\bar{p}}}-\frac{[t^{-1}(\gamma'\circ \f)^2+\gamma''\circ \f]|\f_p|^2}{u_{p\bar{p}}}-\frac{\gamma'\circ \f\f_{p\bar{p}}}{u_{p\bar{p}}}.$$

Since $u_{p\bar{p}}=\f_{p\bar{p}}+h_{p\bar{p}}$ near $(t_0,z_0)$, then at $(t_0,z_0)$
\begin{eqnarray*}
\sum_p \frac{\gamma'\circ \f\f_{p\bar{p}}}{u_{p\bar{p}}}&=&n\gamma'\circ \f-\sum_p \frac{\gamma'\circ \f h_{p\bar{p}}}{u_{p\bar{p}}}\\
 &\leq& n\gamma'\circ \f-\lambda\sum_p \frac{\gamma'\circ \f }{u_{p\bar{p}}} , 
\end{eqnarray*}
 with $\lambda\leq h_{p\bar{p}}$, on $[0,T]\times X$ for all $ p=1,\ldots,n$.

\medskip
Moreover, assume that the holomorphic bisectional curvature of $\omega$ is bounded from below by a constant $-B\in \R$ on X, then at $(t_0,z_0)$
$$t\sum_{j,k,p}\frac{-g_{j\bar{k}p\bar{p}}\f_j\f_{\bar{k}}}{\beta u_{p\bar{p}}}=t\sum_{j,k,p}\frac{R_{j\bar{k}p\bar{p}}\f_j\f_{\bar{k}}}{\beta u_{p\bar{p}}}\geq -Bt\sum_{p} \frac{1}{u_{p\bar{p}}},$$
therefore
\begin{eqnarray*}
\Delta_{\omega_{t_{0}}}K&\geq &(\lambda\gamma'\circ \f-tB)\sum_p \frac{1}{u_{p\bar{p}}}+2tRe\sum_{j,p} \frac{\f_{p\bar{p}j}\f_{\bar{j}}}{\beta u_{p\bar{p}}}\\
&&+\frac{t}{\beta}\sum_{j,p}\frac{ |\f_{jp}-\sum_k g_{k\bar{j}\bar{p}}\f_k  |^2+| \f_{j\bar{p}} -\sum_kg_{k\bar{j}p}\f_k|^2}{\beta u_{p\bar{p}}}\\
&& - [t^{-1}(\gamma'\circ \f)^2+
\gamma''\circ \f]\sum_p\frac{|\f_p|^2}{u_{p\bar{p}}}-n\gamma'\circ \f.
\end{eqnarray*}
By the maximum principle, at $(t_0,z_0)$
\begin{align*}
0\leq \left(\frac{\partial}{\partial t}-\Delta_{\omega_t}\right)K
\end{align*}
hence,
\begin{eqnarray}\label{ineq 1}
0&\leq &\log \beta-\gamma'\circ \f\dot{\f}-(\lambda\gamma'\circ \f-tB)\sum_p \frac{1}{u_{p\bar{p}}}+t\frac{\beta'}{\beta}-2tRe\sum_{j,p} \frac{\f_{p\bar{p}j}\f_{\bar{j}}}{\beta u_{p\bar{p}}}\nonumber\\
&&-\frac{t}{\beta}\sum_{j,p}\frac{ |\f_{jp}-\sum_k g_{k\bar{j}\bar{p}}\f_k  |^2+| \f_{j\bar{p}} -\sum_kg_{k\bar{j}p}\f_k|^2}{\beta u_{p\bar{p}}}\\
&& + [t^{-1}(\gamma'\circ \f)^2+
\gamma''\circ \f]\sum_p\frac{|\f_p|^2}{u_{p\bar{p}}}+n\gamma'\circ \f. \nonumber
\end{eqnarray}
We will simplify (\ref{ineq 1}) to get a  bound for $\beta$ at $(t_0,z_0)$. 

\medskip
\noindent{\bf Claim 1.} {\it  There exist $C_1>0$ depending on $\sup |\f_0|$ and $C_2,C_3, C_4>0$ only depending on $h$ such that
\begin{equation*}
t\frac{\beta'}{\beta} -2tRe\sum_{j,p} \frac{u_{p\bar{p}j}\f_{\bar{j}}}{\beta u_{p\bar{p}}}<C_1t+ C_2 t\sum_p\frac{1}{u_{p\bar{p}}},
\end{equation*} and
$$-\frac{t}{\beta}\sum_{j,p}\frac{ |\f_{jp}-\sum_k g_{k\bar{j}\bar{p}}\f_k|^2}{\beta u_{p\bar{p}}} + t^{-1}(\gamma'\circ \f)^2\sum_p\frac{|\f_p|^2}{u_{p\bar{p}}}\leq C_3\gamma'\circ \f + C_4\sum_p\frac{|\f_p|^2}{u_{p\bar{p}}}.$$}

\begin{proof}[Proof of Claim 1]
For the first one, we note that near $(t_0,z_0)$
$$\log\det(u_{p\bar{q}})=\dot{
\f}+F(t,z,\f)+ \log\Omega,$$
hence using 
$$\frac{d}{ds}\det A=A^{\bar{j}i}\left( \frac{d}{ds}A_{i\bar{j}} \right) \det A$$
we have at $(t_0,z_0)$
$$u^{p\bar{p}}u_{p\bar{p}j}=\frac{u_{p\bar{p}j}}{u_{p\bar{p}}}=(\dot{\f}+F(t,z,\f) +\log \Omega )_j.$$
Therefore
\begin{eqnarray*}
2tRe\sum_{j,p} \frac{\f_{p\bar{p}j}\f_{\bar{j}}}{\beta u_{p\bar{p}}}&=&2tRe\sum_{j,p} \frac{(u_{p\bar{p}j}-h_{p\bar{p}j})\f_{\bar{j}}}{\beta u_{p\bar{p}}}\\
&=&\frac{2t}{\beta}Re\sum_j \large( \dot{\f}+F(t,z,\f) +\log\Omega\large)_j\f_{\bar{j}}-2tRe\sum_{j,p} \frac{h_{p\bar{p}j}\f_{\bar{j}}}{\beta u_{p\bar{p}}} \\
&=& \frac{2t}{\beta} Re\sum_{j} (\dot{\f}_j\f_{\bar{j}})+\frac{2t}{\beta}Re\left( (F(t,z,\f)+\log\Omega)_j+\frac{\partial F}{\partial s}\f_j \right)\f_{\bar{j}}\\
&&-2tRe\sum_{j,p} \frac{h_{p\bar{p}j}\f_{\bar{j}}}{\beta u_{p\bar{p}}}.
\end{eqnarray*}
In addition, at $(t_0,z_0)$
\begin{align*}
t\frac{\beta'}{\beta}&=\frac{t}{\beta} \sum_{j,k} g^{j\bar{k}}(\dot{\f}_j\f_{\bar{k}} +\f_j\dot{\f}_{\bar{k}}) \\
&= \frac{2t}{\beta}Re(\dot{\f}_j\f_{\bar{j}}),
\end{align*}
we infer that
\begin{eqnarray*}
t\frac{\beta'}{\beta} -2tRe\sum_{j,p} \frac{u_{p\bar{p}j}\f_{\bar{j}}}{\beta u_{p\bar{p}}}&=&- \frac{2t}{\beta}Re\sum_j\bigg( F_j(t,z,\f) +(\log\Omega)_j\bigg)\f_{\bar{j}}-\frac{2t}{\beta}\sum_j \frac{\partial F}{\partial s}|\f_j|^2\\
&&+2tRe\sum_{j,p} \frac{h_{p\bar{p}j}\f_{\bar{j}}}{\beta u_{p\bar{p}}}.
\end{eqnarray*}
We may assume that $\log\beta>1$ so that 
$$\frac{|\f_{\bar{j}}|}{\beta}<C$$
Since $\frac{\partial F}{\partial s}\geq 0$, there exist $C_1>0$ depending on $\sup |\f_0|$ and $C_2$ depending on $h$  such that
\begin{equation}\label{5}
t\frac{\beta'}{\beta} -2tRe\sum_{j,p} \frac{u_{p\bar{p}j}\f_{\bar{j}}}{\beta u_{p\bar{p}}}<C_1t+ C_2 t\sum_p\frac{1}{u_{p\bar{p}}}.
\end{equation}
\medskip
We now estimate
$$-\frac{t}{\beta}\sum_{j,p}\frac{ |\f_{jp}-\sum_k g_{k\bar{j}\bar{p}}\f_k  |^2}{\beta u_{p\bar{p}}} + t^{-1}(\gamma'\circ \f)^2\sum_p\frac{|\f_p|^2}{u_{p\bar{p}}}.$$
It follows from (\ref{4}) and (\ref{bp}) that
\begin{align*}
 &\beta_p=-g_{k\bar{j}\bar{p}}\f_j\f_{\bar{k}}+g^{j\bar{k}}\f_{jp}\f_{\bar{k}}+g^{j\bar{k}}\f_{j}\f_{\bar{k}p}.\\
 & t\beta_p=\beta \gamma'\circ \f\f_p
\end{align*}
then,
$$ \sum_{j} (\f_{jp}-\sum_{k}g_{k\bar{j}\bar{p} }\f_k )\f_{\bar{j}}= t^{-1}\gamma'\circ \f\beta\f_p-\sum_j \f_{j\bar{p}}\f_{\bar{j}}=t^{-1}\gamma'\circ \f\beta\f_p-u_{p\bar{p}}+\sum
_{j}h_{j\bar{p}}\f_{\bar{j}}$$
Hence at $(t_0,z_0)$, using $\log \beta>1$, we have
\begin{align*}
\frac{t}{\beta}\sum_{j,p}\frac{ |\f_{jp}-\sum_k g_{k\bar{j}\bar{p}}\f_k  |^2}{\beta u_{p\bar{p}}}&\geq \frac{t}{\beta^2}\sum_{p}\frac{|\sum_{j} (\f_{jp}-\sum_{k}g_{k\bar{j}\bar{p} }\f_k )\f_{\bar{j}}|^2}{u_{p\bar{p}}}\\
&=\frac{t}{\beta^2}\sum_p \frac{|t^{-1}\gamma'\circ \f\beta\f_p-u_{p\bar{p}}+\sum
_{j}h_{j\bar{p}}\f_{\bar{j}}|^2}{u_{p\bar{p}}}\\ 
&\geq t^{-1}(\gamma'\circ \f)^2\sum_p \frac{|\f_p|^2}{u_{p\bar{p}}}+ \frac{1}{\beta}\gamma' \circ \f2Re\sum_{i,p}\frac{h_{j\bar{p}}\f_{\bar{j}}\f_{\bar{p}}}{u_{p\bar{p}}} -C_3\gamma'\circ\f,\\
&\geq t^{-1}(\gamma'\circ \f)^2\sum_p \frac{|\f_p|^2}{u_{p\bar{p}}} -C_3\gamma'\circ\f -C_4\sum_p\frac{|\f_p|^2}{u_{p\bar{p}}},
\end{align*}
here $C_3$ and $C_4$ only depend on $h$. This completes Claim 1.
\end{proof}

\medskip
We now choose 
$$\gamma(s)=As-\left(\frac{1}{A}+C_4\right)s^2$$
with $A$ so large that $\gamma'=A-2\left(\frac{1}{A}+C_4\right)s>0$ and $\gamma''=-2/A-2C_4<0$ for all $s\leq\sup_{[0,T]\times X} \f_t $. From Lemma \ref{bound f'} we have $\dot{\f}\geq C_0+n\log t$, where $C_0$ depends on $Ocs_X \f_0$. Combining this with (\ref{ineq 1}) and Claim 1, we obtain
\begin{align}\label{main inequality}
0\leq -\frac{2}{A}\sum_p\frac{|\f_p|^2}{u_{p\bar{p}}}-(\lambda \gamma'\circ \f-Bt-C_2t)\sum_p \frac{1}{u_{p\bar{p}}}+ \log \beta+ C'\gamma'\circ \f+C_1t,
\end{align}
where $C_1,C_2,C'$ depend on $\sup_X|\f_0|$ and $h$.
If A is chosen sufficiently large, we have a constant $C_5>0$ such that 
\begin{equation}\label{ineq 2}
\sum_p\frac{1}{u_{p\bar{p}}} +\sum_p \frac{|\f_p|^2}{u_{p\bar{p}}}\leq C_5\log \beta,
\end{equation}
since otherwise (\ref{main inequality}) implies that $\beta$ is  bounded. So we get
$(u_{p\bar{p}})^{-1}\leq C_5\log\beta$ for $1\leq p\leq n$. It follows  from Lemma \ref{bound f' above} we have at $(t_0,z_0)$ $$ \prod_p u_{p\bar{p}}=e^{-\dot{\f}_t+F(t,x,\f_t)}\leq C_6, $$
where $C_6$ depends on $\sup_X|\f_0|,\e$.  Then we get
$$u_{p\bar{p}}\leq C_6(C_5\log \beta)^{n-1},$$
so from (\ref{ineq 2}) we have
$$\beta=\sum_p |\f_p|^2\leq C_6(C_5\log \beta)^n,$$
hence $\log \beta< C_7$ at $(t_0,z_0)$. This shows that $\beta=|\nabla\f(z)|^2_\omega<e^{C/t}$ for some $C$ depending on $\sup|\f_0|$ and $\e$.
\end{proof}

\subsection{Bounding $\Delta \f_t $} 
We now estimate $\Delta \f$. The estimate on $|\nabla \f|^2_{\omega}$ is needed here. The argument follows from \cite{GZ13,To16} but there are difficulties in using this approach because of torsion terms that need to be controlled (see also \cite{TW10} for similar computation for the elliptic Monge-Amp\`ere equation).
\begin{lem}\label{bound delta}
Fix $\e>0$. There exist constants $A$ and $C$ only depending on $\e$ and $\sup_X\f_0$ such that for all $0\leq t\leq T-\e$,
$$0\leq t\log\tr_\omega(\omega_{t+\e})\leq AOsc_X(\f_\e)+C+[C-n\log \e+ AOsc_X(\f_\e)]t.$$
\end{lem}

\begin{proof}
We first denote by $C$ a uniform constant only depending on $\e$ and $\sup_X \f_0$.

\medskip
Define
$$P=t\log\tr_{\omega}(\omega_{t+\e}) -A \f_{t+\e},$$
and $$u=\tr_{\omega}(\omega_{t+\e}),$$
with $A>0$ to be chosen latter. Set $\Delta_t:=\Delta_{\omega_{t+\e}}$, then
\begin{align*}
\frac{\partial}{\partial t}P&=\log u+t\frac{\dot{u}}{u}-A\dot{\f}_{t+\e},\\
\Delta_t P&=t\Delta_t\log u-A\Delta_t\f_{t+\e}
\end{align*}
hence
\begin{equation}\label{heat operator}
\left( \frac{\partial}{\partial t}-\Delta_t \right) P=\log u+ t\frac{\dot{u}}{u} - A\dot{\f}_{t+\e}-t\Delta_t \log u+A\Delta_t\f_{t+\e}.
\end{equation}
First, we have 
\begin{equation}\label{delta f_t+e}
A\Delta_t\f_{t+\e} =An-A\tr_{\omega_{t+\e}}(\theta_{t+\e})\leq An-\frac{A}{2}\tr_{\omega_{t+\e}}(\omega).
\end{equation}

Suppose $P$ attains its maximum at $(t_0,z_0)$. If $t_0=0$, we get the desired inequality. We now assume that $P(t,x)$ attains its maximum at $(t_0,x_0)$ with $t_0>0$. 

\medskip
It follows from Proposition \ref{bound f'}, Proposition \ref{bound f' above} and Theorem \ref{bound oscillation} that $\dot{\f}_{t+\e}$ depends on $\e$ and $\sup_X\f_0$, hence $$\bigg|\log \frac{\omega_t^n}{\omega^n}\bigg|\leq C.$$
Combine with the inequality 
$$\bigg(\frac{\tilde{\omega}^n}{\omega^n} \bigg)^{\frac{1}{n}}\leq\frac{1}{n}\tr_\omega \tilde{\omega}\quad (\text{cf. \cite[Lemma 4.1.1.]{BG13}}),$$ we infer that $$\tr_\omega\omega_{t+\e}\geq C^{-1},\quad\tr_{\omega_{t+\e}}\omega\geq C^{-1},$$ 
\begin{equation}\label{lower bound of trace}
\tr_\omega\omega_{t+\e}\leq C \tr_\omega\omega_{t+\e}\tr_{\omega_{t+\e}}\omega,\quad \tr_{\omega_{t+\e}}\omega\leq C \tr_\omega\omega_{t+\e}\tr_{\omega_{t+\e}}\omega.
\end{equation}

\medskip
Denoting $\tilde{g}(t,x)=g_{t+\e}(x)$ and using the local coordinate system (\ref{GL local coor}) at $(t_0,x_0)$, we have 
\begin{eqnarray*}
\Delta_t\tr_\omega\omega_{t+\e}&=& \tilde{g}^{i\bar{j}}\partial_i\partial_{\bar{j}} (g^{k\bar{l}} \tilde{g}_{k\bar{l}}) \\
&=&\sum\tilde{g}^{i\bar{i}} \tilde{g}_{k\bar{k}i\bar{i}}-2Re(\sum_{i,j,k} g_{j\bar{k}\bar{i}}\tilde{g}_{k\bar{j}i})+\sum\tilde{g}^{i\bar{i}} g_{j\bar{k}i}g_{k\bar{j}\bar{i}}\tilde{g}_{k\bar{k}}\\
&&+ \sum\tilde{g}^{i\bar{i}}g_{k\bar{j}i}g_{j\bar{k}\bar{i}}\tilde{g}_{k\bar{k}} -\sum \tilde{g}^{i\bar{i}}g_{k\bar{k}\bar{i}i}\tilde{g}_{k\bar{k}}\\
&\geq&\sum_{i,k} \tilde{g}^{i\bar{i}}\tilde{g}_{k\bar{k}i\bar{i}}-2Re(\sum_{i,j,k}\tilde{g}^{i\bar{i}}g_{j\bar{k}\bar{i}}\tilde{g}_{k\bar{j}i} )-C\tr_{\omega}\omega_{t+\e}\tr_{\omega_{t+\e}\omega},
\end{eqnarray*}
where the last inequality comes from $(\ref{lower bound of trace})$.
Since \begin{eqnarray}
\Delta_\omega \dot{\f}=\tr_\omega(Ric(\Omega)-Ric(\omega_{t+\e})) +\Delta_\omega F(t,z,\f_{t+\e})
\end{eqnarray}
and
$$\tr_\omega Ric(\omega_{t+\e})=\sum_{i,k}\tilde{g}^{i\bar{i}}(-\tilde{g}_{i\bar{i}k\bar{k}}+\tilde{g}^{j\bar{j}} \tilde{g}_{i\bar{j}k}\tilde{g}_{j\bar{i}\bar{k}}),$$
we have
\begin{eqnarray}
\sum_{i,k}\tilde{g}^{i\bar{i}}\tilde{g}_{i\bar{i}k\bar{k}}=\sum_{i,j,k} \tilde{g}^{i\bar{i}}\tilde{g}^{j\bar{j}}\tilde{g}_{i\bar{j}k}\tilde{g}_{j\bar{i}\bar{k}}-\tr_\omega\omega_{t+\e}.
\end{eqnarray}
Therefore
\begin{eqnarray*}
\Delta_t\tr_\omega\omega_{t+\e}&\geq &\sum_{i,j,k} \tilde{g}^{i\bar{i}}\tilde{g}^{j\bar{j}}\tilde{g}_{i\bar{j}k}\tilde{g}_{j\bar{i}\bar{k}}-\tr_\omega Ric(\omega_{t+\e})\\
&&-2Re(\sum_{i,j,k}\tilde{g}^{i\bar{i}}g_{j\bar{k}\bar{i}}\tilde{g}_{k\bar{j}i} )-C_4\tr_\omega\omega_{t+\e}\tr_{\omega_{t+\e}}\omega
\end{eqnarray*}
Now we have
\begin{eqnarray*}
|2Re(\sum_{i,j,k}\tilde{g}^{i\bar{i}}g_{j\bar{k}\bar{i}}\tilde{g}_{k\bar{j}i} )|&\leq& \sum_{i}\sum_{j \neq k} (\tilde{g}^{i\bar{i}}\tilde{g}^{j\bar{j}}\tilde{g}_{i\bar{j}k} \tilde{g}_{j\bar{i}\bar{k}}+\tilde{g}^{i\bar{i}}\tilde{g}_{j\bar{j}}g_{j\bar{k}}g_{k\bar{j}\bar{k}})\\
&\leq& \sum_{i}\sum_{j \neq k} \tilde{g}^{i\bar{i}}\tilde{g}^{j\bar{j}}\tilde{g}_{i\bar{j}k} \tilde{g}_{j\bar{i}\bar{k}} +C\tr_\omega\omega_{t+\e}\tr_{\omega_{t+\e}}\omega
\end{eqnarray*}
It follows that 
\begin{equation}\label{delta tr}
\Delta_t\tr_\omega\omega_{t+\e} \geq \sum_{i,j}\tilde{g}^{i\bar{i}}\tilde{g}^{j\bar{j}}\tilde{g}_{i\bar{j}j} \tilde{g}_{j\bar{i}\bar{j}}-\tr_\omega Ric(\omega_{t+\e})-C\tr_\omega\omega_{t+\e}\tr_{\omega_{t+\e}}\omega.
\end{equation}
We now claim that 
\begin{eqnarray}\nonumber
\frac{|\partial\tr_\omega(\omega_{t+\e})|^2_{\omega_{t+\e}}}{(\tr_{\omega}\omega_{t+\e})^2}&\leq &\sum_{i,j}\tilde{g}^{i\bar{i}}\tilde{g}^{j\bar{j}} \tilde{g}_{i\bar{j}j}\tilde{g}_{j\bar{i}\bar{j}}+\frac{1}{t}  A^2|\nabla \f|^2\\ \label{gradient of trace}
&&+ \left(1+\frac{1}{t}\right)\frac{C\tr_{\omega_{t+\e}}\omega}{(\tr_{\omega}\omega_{t+\e})^2} +C\tr_{\omega_{t+\e}}\omega.
\end{eqnarray}
By computation,
\begin{eqnarray*}
\frac{|\partial\tr_\omega(\omega_{t+\e})|^2_{\omega_{t+\e}}}{(\tr_{\omega}\omega_{t+\e})^2}&=& \frac{1}{(\tr_{\omega}\omega_{t+\e})^2} \sum_{i,j,k}\tilde{g}^{i\bar{i}} \tilde{g}_{j\bar{j}i}\tilde{g}_{k\bar{k}\bar{i}} \\
&=& \frac{1}{(\tr_{\omega}\omega_{t+\e})^2}\sum_{i,j,k} \tilde{g}^{i\bar{i}}(h_{j\bar{j}i}-h_{i\bar{j}j}+\tilde{g}_{i\bar{j}j})(h_{k\bar{k}\bar{i}}-h_{k\bar{i}\bar{k}}+\tilde{g}_{k\bar{i}\bar{k}}) \\
&=&\frac{1}{(\tr_{\omega}\omega_{t+\e})^2}\sum_{i,j,k} \tilde{g}^{i\bar{i}}(T_{ij\bar{j}}+\tilde{g}_{i\bar{j}j})(\bar{T}_{ik\bar{k}}+\tilde{g}_{k\bar{i}\bar{k}}) \\
&=&\frac{1}{(\tr_{\omega}\omega_{t+\e})^2}\bigg(\sum_{i,j,k}\tilde{g}^{i\bar{i}}\tilde{g}_{i\bar{j}j}\tilde{g}_{k\bar{i}\bar{k}} + \sum_{i,j,k}\tilde{g}^{i\bar{i}}T_{ij\bar{j}}\bar{T}_{ik\bar{k}} +2Re\sum_{i,j,k }\tilde{g}^{i\bar{i}} T_{ij\bar{j}} \tilde{g}_{k\bar{i}\bar{k}}\bigg)
\end{eqnarray*}
where $T_{ij\bar{j}}=h_{j\bar{j}i}-h_{i\bar{j}j}$.

\medskip
It follows from the Cauchy-Schwarz inequality that
\begin{eqnarray}
\frac{1}{\tr_{\omega}\omega_{t+\e}}\sum_{i,j,k}\tilde{g}^{i\bar{i}}\tilde{g}_{i\bar{j}j}\tilde{g}_{k\bar{i}\bar{k}} \leq \sum_{i,j}\tilde{g}^{i\bar{i}}\tilde{g}^{j\bar{j}} \tilde{g}_{i\bar{j}j}\tilde{g}_{j\bar{i}\bar{j}}.
\end{eqnarray}
For the second term, we have
$$\sum_{i,j,k}\tilde{g}^{i\bar{i}}T_{ij\bar{j}}\bar{T}_{ik\bar{k}} \leq C\tr_{\omega_{t+\e}}\omega.$$
Now at the maximum point $(t_0,z_0), t_0>0$, we have $\nabla P=0$, hence $$A\dot{\f}_{\bar{i}}=t\frac{u_{\bar{i}}}{u}=\frac{t}{u}\sum_{k}g_{k\bar{k}\bar{i}}.$$
Since $\tilde{g}_{k\bar{i}\bar{k}}=\tilde{g}_{k\bar{k}\bar{i}}-\bar{T}_{ik\bar{k}}$, we have

 \begin{eqnarray*}
 \bigg| \frac{2}{(\tr_{\omega}\omega_{t+\e})^2} Re\sum_{i,j,k }\tilde{g}^{i\bar{i}} T_{ij\bar{j}} \tilde{g}_{k\bar{i}\bar{k}} \bigg| &=&\frac{2}{(\tr_{\omega}\omega_{t+\e})^2} \bigg| Re\sum_{i,j,k }\tilde{g}^{i\bar{i}} T_{ij\bar{j}} \tilde{g}_{k\bar{k}\bar{i}} +  Re\sum_{i,j,k }\tilde{g}^{i\bar{i}} T_{ij\bar{j}}\bar{T}_{ik\bar{k}} \bigg| \\
 &\leq& \frac{2}{\tr_{\omega}\omega_{t+\e}} \bigg|\frac{A}{t} Re\sum_{i,j }\tilde{g}^{i\bar{i}} T_{ij\bar{j}}\f_{\bar{i}} \bigg| +C\tr_{\omega_{t+\e}}\omega\\
 &\leq& \frac{1}{t} \bigg( A^2|\nabla \f|_\omega^2+ \frac{C\tr_{\omega_{t+\e}}\omega}{(\tr_{\omega}\omega_{t+\e})^2}  \bigg)+C\tr_{\omega_{t+\e}}\omega.
 \end{eqnarray*}
 Combining all of these inequalities we obtain  (\ref{gradient of trace}).
 
 \medskip
 It now follows from (\ref{delta tr}) and (\ref{gradient of trace}) that
 $$\Delta_t\log \tr_{\omega}\omega_{t+\e}\geq  -\frac{\tr_\omega Ric(\omega_{t+\e})}{\tr_\omega\omega_{t+\e}}  -  \frac{1}{t}  A^2|\nabla_\omega \f|^2- \bigg(1+\frac{1}{t}  \bigg)\frac{C\tr_{\omega_{t+\e}}\omega}{(\tr_{\omega}\omega_{t+\e})^2}-C\tr_{\omega_{t+\e}}\omega. $$
 
Moreover,
\begin{align*}
\frac{t\dot{u}}{u}&=\frac{t}{u} \bigg[\Delta_\omega\left( \log \omega_{t+\e}^n/\omega^n -\log\Omega/\omega^n -F(t,z,\f_{t+\e}) \right) +\tr_{\omega}\dot{\theta}_t\bigg],\\
&=\frac{t}{u}\bigg[-\tr_{\omega}(Ric\ \omega_{t+\e})+tr_\omega(\dot{\theta}_t+Ric\ \omega)-\Delta_\omega \left(F(t,z,\f)+\log\Omega/\omega^n\right)\bigg],
\end{align*}
It follows from Proposition \ref{bound grad}, $\tr_{\omega_{t+\e}}\omega\tr_{\omega}\omega_{t+\e}\geq n$ and $\tr_{\omega}\omega_{t+\e}\geq C^{-1}$ that
\begin{equation}\label{ineq 3}
-t\Delta_t \log u+ \frac{t\dot{u}}{u}\leq Ct\tr_{\omega_{t+\e}}(\omega)-t
\frac{\Delta_\omega \big[F(t,z,\f)+\log\Omega/\omega^n\big]}{\tr_{\omega}(\omega_{t+\e})}+C.
\end{equation}

Now $$\Delta_\omega F(t,z,\f_{t+\e})=\Delta_\omega F(z,.)+2Re\bigg[ g^{j\bar{k}}\left(\frac{\partial F}{\partial s}\right)_{j}\f_{\bar{k}}\bigg]+\frac{\partial F}{\partial s}\Delta_\omega \f +\frac{\partial^2 F}{\partial s^2}|\nabla \f|^2_\omega.$$
Therefore 
$$\big|\Delta_\omega\big( F(t,z,\f_{t+\e})+\log\Omega/\omega^n\big)\big|\leq C+C|\nabla\f|^2_\omega +C\tr_\omega\omega_{t+\e}.$$
Then we infer 
$$-\frac{\Delta_\omega [F(t,z,\f)+\log\Omega/\omega^n]}{\tr_{\omega}(\omega_{t+\e})}\leq \frac{1}{n}\tr_{\omega_{t+\e}}(\omega)(C+C|\nabla \f|^2_\omega)+C,$$
so from Proposition \ref{bound grad} and (\ref{ineq 3}) we have
\begin{equation}\label{ineq 4}
-t\Delta_t \log u+ \frac{t\dot{u}}{u}\leq Ct\tr_{\omega_{t+\e}}(\omega) +C.
\end{equation}

\medskip
Moreover, the inequalities $(n-1)\log x\leq x+C_n$ and 
 $$\tr_{\omega}\tilde{\omega}\leq \bigg(\frac{\tilde{\omega}^n}{\omega^n}\bigg)(\tr_ {\tilde{\omega}}\omega)^{n-1},$$
 for any two positive $(1,1)$-froms $\omega$ and $\tilde{\omega}$, imply that
\begin{align}\nonumber
\log u&=\log \tr_{\omega}(\omega_{t+\e})\leq \log \left( n\left( \frac{\omega_{t+\e}^n}{\omega^n}\right) \tr_{\omega_{t+\e}}(\omega)^{n-1} \right)\\ \nonumber
&=\log n + \dot{\f}_{t+\e}+F(t,z,\f) +(n-1)\log \tr_{\omega_{t+\e}}(\omega)\\ 
&\leq \dot{\f}_{t+\e}+ \tr_{\omega_{t+\e}}(\omega) +C. \label{ln u}
\end{align}

It follows from (\ref{heat operator}), (\ref{delta f_t+e}), (\ref{ineq 4}) and (\ref{ln u}) that
$$\left( \frac{\partial}{\partial t}-\Delta_t \right) P\leq C-(A-1)\dot{\f}_{t+\e}+ [Ct+1-A/2]\tr_{\omega_{t+\e}}\omega. $$

\medskip We choose $A$ sufficiently large such that $Ct+1-A/2<0$. Applying Proposition \ref{bound f'},
$$\left( \frac{\partial}{\partial t}-\Delta_t \right) P\leq C-(A-1)(n\log \e-AOsc_X\f_\e-C).$$
Now suppose $P$ attains its maximum at $(t_0,z_0)$. If $t_0=0$, we get the desired inequality. Otherwise, at $(t_0,z_0)$
$$0\leq \left( \frac{\partial}{\partial t}-\Delta_t \right) P\leq C-(A-1)(n\log \e-AOsc_X\f_\e-C).$$
Hence we get $$t\log\tr_\omega(\omega_{t+\e})\leq AOsc_X(\f_\e)+C+[C-n\log \e+ AOsc_X(\f_\e)]t,$$
as required.
\end{proof}
\subsection{Higher order estimates}
For the higher order estimates, one can follow  \cite{SzTo}  (see \cite{Nie} for  its version on Hermitian manifolds)  by bounding $$S=g_\f^{i\bar{p}}g_\f^{q\bar{j}}g_\f^{k\bar{r}}\f_{i\bar{j}k}\f_{\bar{p}q\bar{r}}\, \text{ and } |Ric(\omega_t)|_{\omega_t},$$
then using the parabolic Schauder estimates to obtain higher order estimates for $\f$. Additionally, we can also combine previous estimates with Evans-Krylov and Schauder estimates \cite[Theorem 1.7]{To16} to get the $C^k$ estimates for all $k\geq 0$.
\begin{thm}\label{full estimates}
For each $\e>0$ and $k\in \N$, there exists $C_k(\e)$ such that
$$||\f||_{\mathcal{C}^k([\e,T]\times X)}\leq C_k(\e).$$ 
\end{thm}
\section{Proof of Theorem A and B}\label{proof}
We now consider the complex Monge-Amp\`ere flow
$$
\hskip-2cm
(CMAF) \hskip2cm \frac{\partial \f_t}{\partial t}=\log  \frac{(\theta_t+dd^c \f_t)^n}{\Omega} -F(t,z,\f),
$$
starting from a $\omega$-psh function $\f_0$ with zero Lelong numbers at all points, where $ F(t,z,s)\in C^\infty([0,T]\times X\times \R,\R)$  with $\frac{\partial F}{\partial s}\geq 0$ and  $ \frac{\partial F}{\partial t}$ is bounded from below. 

\subsection{Convergence in $L^1$}\label{conv L1} We fisrt approximate $\f_0$ by a decreasing sequence $\f_{0,j}$ of smooth $\omega$-psh fuctions (see \cite{BK07}). Denote by $\varphi_{t,j}$ the smooth family of $\theta_t$-psh functions satisfying on $[0,T]\times X$
$$
\frac{\partial \f_t}{\partial t}=\log  \frac{(\theta_t+dd^c \f_t)^n}{\Omega} -F(t,z,\f)
$$
with initial data $\f_{0,j}$.
\medskip

It follows from the maximum principle \cite[Proposition 1.5]{To16}  that $j\mapsto \f_{j,t}$ is non-increasing. Therefore we can set
$$\f_t(z):=\lim_{j\rightarrow +\infty} \f_{t,j}(z).$$
Thanks to Lemma \ref{bound from below} the function $t\mapsto \sup_X \f_{t,j}$  is uniformly bounded, hence $\f_t$ is a well-defined $\theta_t$-psh function. Moreover, it follows from Theorem \ref{full estimates} that $\f_t$ is also smooth in $(0,T]\times X$ and satisfies
$$\frac{\partial \f_t}{\partial t}=\log  \frac{(\theta_t+dd^c \f_t)^n}{\Omega} -F(t,z,\f).$$

Observe that $(\f_t)$ is relatively compact in $L^1(X)$ as $t\rightarrow 0^+$, we now show  that $\f_t\rightarrow\f_0$ in $L^1(X)$ as $t\searrow 0^+$.
\medskip

First, let $\f_{t_k}$ is a subsequence of $(\f_t)$ such that $\f_{t_k}$ converges to some function $\psi$ in $L^1(X)$ as $t_k\rightarrow 0^+$. By the properties of plurisubharmonic functions, for all $z\in X$
$$\limsup_{t_k\rightarrow 0} \f_{t_k}(z)\leq \psi(z),$$
with equality almost everywhere. We infer that for almost every $z\in X$
$$\psi(z)=\limsup_{t_k\rightarrow 0}\f_{t_k}(z)\leq \limsup_{t_k\rightarrow 0} \f_{t_k,j}(z)=\f_{0,j}(z),$$
by continuity of $\f_{t,j}$ at $t=0$. Thus $\psi\leq \f_0$ almost everywhere.

\medskip
Moreover, it follows from Lemma \ref{bound from below} that
$$\f_t(z)\geq (1-2\beta t)\f_0(z)+t u(z)+n(t\log t-t)+ At, $$
with $u$ continuous, so
$$\f_0\leq \liminf_{t\rightarrow 0}\f_t.$$
Since $\psi\leq \f_0$ almost everywhere, we get $\psi=\f_0$ almost everywhere, so $\f_t\rightarrow \f_0$ in $L^1$.
\subsection{Uniform convergence}
If the initial condition $\f_0$ is continuous then by \cite[Proposition 1.5]{To16} we infer that $\f_t\in C^0([0,T]\times X)$, hence $\f_t$ uniformly converges to $\f_0$ as $t\rightarrow 0^+$. 

\subsection{Uniqueness and stability of solution}\label{uniqueness and stability}
We now study the uniqueness and stability for the complex Monge-Amp\`ere flow
\begin{equation}\label{reduced equation}
\frac{\partial \f_t}{\partial t}=\log  \frac{(\theta_t+dd^c \f_t)^n}{\Omega} -F(t,z,\f),
\end{equation}

where $ F(t,z,s)\in C^\infty([0,T]\times X\times \R,\R)$ satisfies
$$\frac{\partial F}{\partial s}\geq0\, \text{ and }\, \left|\frac{\partial F}{\partial t}\right|\leq C',$$
for some constant $C'>0$. 

\medskip
The uniqueness of solution follows directly from the same result in the K\"ahler setting \cite{To16} 

\begin{thm}\label{uniqueness}
Suppose $\psi$ and $\f$ are two solutions of (\ref{reduced equation}) with $\f_0\leq\psi_0$, then 
$\f_t\leq\psi_t.$ In particular, the equation (\ref{reduced equation}) has a unique solution. 
\end{thm}

The stability result also comes from the same argument as in \cite{To16}. The difference is that we use Theorem \ref{DK} instead of the one in K\"ahler manifolds. 
\begin{thm}\label{qualitative stab}
Fix $\varepsilon>0$. Let $\varphi_{0,j}$ be a sequence of $\omega$-psh functions with zero Lelong number at all points, such that $\f_{0,j}\rightarrow \f_0$ in $L^1(X)$. Denote by $\varphi_{t,j}$ and $\f_j$ the solutions of (\ref{reduced equation}) with the initial condition $\f_{0,j}$ and $\f_0$ respectively. Then
$$\f_{t,j}\rightarrow \f_{t}\  \text{ in }\  C^\infty ([\e, T]\times X)\ \text{ as }\  j\rightarrow +\infty.$$

\medskip
Moreover, if $\f,\psi\in C^\infty((0,T]\times X)$  are solutions of $(CMAF)$ with continuous initial data $\f_0$ and $\psi_0$, then 
\begin{equation}\label{stability ineq}
||\f-\psi||_{C^{k}([\e,T]\times X)}\leq C(k,\e)||\f_0-\psi_0||_{L^\infty(X)
}.
\end{equation} 
\end{thm}
\begin{proof}
We use the techniques in Section \ref{a priori estm} to obtain estimates of $\f_{t,j}$ in $C^k([\e,T]\times X)$ for all $k\geq 0$. In particular, for the $C^0$ estimate, we need to have uniform bound for 
$H_{t,j}=\exp(\dot{\f}_{t,j}+F)$ in order to use Theorem \ref{DK}. By Lemma \ref{bound f' above} we have
$$H_{t,j}=\exp(\dot{\f}_{t,j}+F)\leq \exp\bigg(\frac{-\f_{0,j}+C}{t}+C'\bigg),$$
where $C,C'$ depend on $\e, \sup_X {\f_{0,j}}$. Since $\f_{0,j}$ converges to $\f_0$ in $L^1$, we have the $\sup_X\f_{0,j}$ is uniformly bounded in term of $\sup_X\f_0$ for all $j$ by the Hartogs lemma, so we can choose $C,C'$ independently of $j$. It follows from \cite[Theorem 0.2 (2)]{DK01} that there is a constant $A(t,\e)$ depending on $t$ and $\e$ such that $||H_{t,j}||_{L^2(X)}$ is uniformly bounded by $A(t,\e)$ for all $t\in [\e,T]$. The rest of the proof is now siminar to \cite[Theorem 4.3]{To16}.
\end{proof}
\section{Chern-Ricci flow and canonical surgical contraction}\label{proof of conjecture} In this section, we give a proof of the conjecture of Tosatti and Weinkove. Let  $(X,\omega_0)$ be a Hermitian manifold. Consider the Chern-Ricci flow on $X$,
\begin{equation}\label{CRF contraction}
\frac{\partial}{\partial t}\omega_t=-Ric(\omega),\qquad  \omega{|_{t=0}}=\omega_0. 
\end{equation}
Denote
$$T:=\sup\{t\geq 0| \exists \psi\in C^\infty(X) \text{ with } \hat{\omega}_X+dd^c\psi>0\},$$
where $\hat{\omega}_X= \omega_0+t\chi,$ with $\chi$ is a smooth $(1,1)$-form representing $- c_1^{BC}(X)$.

\medskip
Assume that there exists a holomorphic map between compact Hermitian manifolds $\pi:X\rightarrow Y$  blowing down an exceptional divisor $E$ on $X$ to one point $y_0\in Y$. In addition, assume that there exists a smooth function $\rho$ on $X$ such that
\begin{equation}\label{condition of cohomology class}
\omega_0-TRic(\omega_0)+dd^c\rho=\pi^*\omega_Y,
\end{equation}
with $T<+\infty$, where $\omega_Y$ is a Hermitian metric on $Y$. In \cite{TW15,TW13}, Tosatti and Weinkove proved that the solution $\omega_t$ to the Chern-Ricci flow (\ref{CRF contraction}) converges in $C^\infty_{loc}(X\setminus E)$ to a smooth Hermitian metric $\omega_T$ on $X\setminus E$.
Moreover, there exists a distance function $d_T$ on $Y$ such that $(Y,d_T)$ is a compact metric space and $(X,g(t))$ converges in the Gromov-Hausdorff sense $(Y,d_T)$ as $t\rightarrow T^-$. Denote by $\omega' $ the push-down of the current $\omega_T$ to $Y$. They conjectured that:
\begin{conj} \cite[Page 2120]{TW13}\label{conjecture}
\begin{enumerate}
\item There exists a smooth maximal solution $\omega_t$ of the Chern-Ricci flow on $Y$ for $t\in (T,T_Y)$ with $ T<T_Y\leq+\infty$ such that $\omega_t$ converges to $\omega'$, as $t\rightarrow T^+$, in $C^{\infty}_{loc}(Y\setminus \{y_0\})$. Furthermore, $\omega_t$ is uniquely determined by $\omega_0$.
\medskip
\item The metric space $(Y,\omega_t)$ converges to $(Y,d_T)$ as $t\rightarrow T^+$ in the Gromov-Hausdorff sense. 
\end{enumerate}
\end{conj}

\medskip
We now prove this conjecture using Theorem A, Theorem B and some arguments in \cite{SW13, TW13}. 
\subsection{Continuing the Chern-Ricci flow}
We prove the first claim in the conjecture showing how to continue the Chern-Ricci flow. 

\medskip
Write $\hat{\omega}= \pi^*\omega_Y =\omega_0-TRic(\omega_0)+dd^c\rho$. Then there is a positive $(1,1)$-current $\omega_T=\hat{\omega}+dd^c\f_T$ for some bounded function $\f_T$. By the same argument in \cite[Lemma 5.1]{SW13} we have 
$$\f_T |_E=constant.$$
Hence there exists a bounded function $\phi_T$ on $Y$ that is smooth on $Y\setminus \{y_0\}$ with $\f_T=\pi^*\phi_T$.

\medskip
We now define a positive current $\omega'$ on $Y$ by
$$\omega'= \omega_Y+dd^c \phi_T\geq 0,$$
which is the push-down of the current $\omega_T$ to $Y$ and is smooth on $Y\setminus \{y_0\}$. By the same argument in \cite[Lemma 5.2]{SW13} we have $\omega'^n/\omega_Y^n\in L^p(Y) $. It follows from \cite[Theorem 5.2]{DK12} that $\phi_T$ is continuous. 

\medskip
We fix a smooth $(1,1)$ form $\chi\in -c_1^{BC}(Y)$ and a smooth volume form  $\Omega_Y$ such that $\chi=dd^c\log \Omega_Y$. Denote
$$T_Y:=\sup\{t>T| \omega_Y+(t-T)\chi>0\}.$$
Fix $T'\in (T,T_Y)$, we have: 
\begin{thm}
There is a unique smooth family of Hermitian metrics $(\omega_t)_{T<t\leq T'}$ on $Y$ satisfying the Chern-Ricci flow
\begin{equation}\label{weak CRF}
\frac{\partial \omega_t}{\partial t}=-Ric(\omega_t), \qquad \omega_t|_{t=T}=\omega', 
\end{equation}
with $\omega_t=\omega_Y+(t-T)\chi+dd^c\phi_t$. Moreover, $\phi_t$ uniformly converges to $\phi_T$ as $t\rightarrow T^+$. 
\end{thm}

\begin{proof}
We can rewrite the flow as the following complex Monge-Amp\`ere flow
\begin{equation}\label{MAF for CRF}
\frac{\partial \phi_t}{\partial t}=\log\frac{(\hat{\omega}_Y+dd^c\phi_t)^n}{\Omega},\qquad \phi |_{t=T}=\phi_T,
\end{equation}
where $\hat{\omega}_Y:=\omega_Y+(t-T)\chi $ and $\phi_T$ is continuous.

\medskip
It follows from Theorem A and Theorem B that there is a unique solution $\phi$ of (\ref{MAF for CRF}) in $C^\infty ((T,T']\times Y)$ such that $ \phi_t $ uniformly converges to $\phi_T$ as $t\rightarrow T^+$.
\end{proof}

\subsection{Backward convergences} 
Once the Chern-Ricci flow can be run from $\omega'$ on  $Y$, we can prove the rest of Conjecture \ref{conjecture} following the idea in \cite[Section 6]{SW13}. 

\medskip
We keep the notation as in \cite{TW15}. Let $h$ be a Hermitian metric on the fibers of the line bundle $[E]$ associated to the divisor $E$, such that for  $\e_0>0$ sufficiently small, we have
\begin{equation}\label{curvature of h}
\pi^* \omega_Y-\e_0 R_h>0, \qquad \text{where} \, \, R_h:=-dd^c\log h.
\end{equation}
Take $s$ a holomorphic section of $[E]$ vanishing along $E$ to order 1. We fix a a coordinate chart $U$ centered at $y_0$, which identities with the unit ball $B\subset \C^n $ via coordinates $z_1,\ldots,z_n$. Then the function $|s|^2_h$ on $X$ is given on $\pi^{-1}\left(B(0,1/2)\right)$ by
$$|s|^2_h(x)=|z_1|^2+\ldots+|z_n|^2:=r^2, \quad \text{ for }\pi(x)=(z_1,\ldots,z_n).$$
Hence, the curvature  $R(h)$ of $h$ is given by
$$R(h):=-dd^c\log (|z_1|^2+\ldots+|z_n|^2).$$ 

\medskip
The crucial ingredient of the proof of the conjecture is the following proposition:
\begin{prop}\label{two key estimates}

The solution $\omega_t$ of (\ref{weak CRF}) is in $C^{\infty}([T,T']\times Y\setminus\{y_0\})$ and there exists $\eta>0$ and a uniform constant $C>0$ such that for $t\in [T,T']$
\begin{enumerate}
\item $\omega_t\leq C\frac{\omega_Y}{\pi_*|s|^2_h},$
\item $\omega_t\leq C\pi_*\left(\frac{\omega_0}{|s|^{2(1-\eta)}} \right)$.
\end{enumerate}
\end{prop}

\medskip
In order to prove this propositon, we use the method in \cite{SW13} to construct a smooth approximant of the solution $\phi_t$ of (\ref{MAF for CRF}).  Denote by  $f_\e$  a family of positive smooth functions $f_\e$ on $Y$ such that it has the form
$$ f_\e(z)=(\e+r^2)^{n-1}, $$
on $B$, hence $f_\e(z)\rightarrow f(z)=r^{2(n-1)}$ as $\e\rightarrow 0$. Moreover, there is a smooth volume form $\Omega_X$ on $X$ with $\pi^*\Omega_Y=(\pi^*f)\Omega_X$.  

\medskip
Observe that $\hat{\omega}_Y(t)-\frac{\e}{T}\omega_Y$ is Hermitian on $Y$  for $t\in [T,T']$ if $\e$ is sufficiently small.  Therefore  $$\theta^\e:=\pi^*(\hat{\omega}_Y-\frac{\e}{T}\omega_Y)+\frac{\e}{T}\omega_0$$
is Hermitian for $\e>0$ sufficiently small.

\medskip
We denote $\psi^\e_t$ the unique smooth solution of the following Monge-Amp\`ere flow  on $X$:
\begin{align}\nonumber
&\frac{\partial \psi^\e}{\partial t}=\log\frac{\left(\theta^\e +dd^c\psi^\e\right)^n}{(\pi^*f_\e)\Omega_X}\\ \label{epsilon MAF}
&\psi^\e|_{t=T}=\f(T-\e),
\end{align}

Define K\"ahler metrics $\omega^\e$ on $[T,T']\times X$ by
\begin{equation}
\omega^\e=\theta^\e+dd^c\psi^\e,
\end{equation}

then 
\begin{equation}\label{example of twisted CRF}
\frac{\partial \omega^\e}{\partial t}=-Ric(\omega^\e)-\eta,
\end{equation}
where $\eta= -dd^c\log((\pi^*f_\e)\Omega_X)+\pi^*\chi= -dd^c\log(\pi^*f_\e)\Omega_X)+dd^c\log((\pi^*f)\Omega_X) .$

\medskip
We claim that 
$\pi_*\psi^\e_t$ converges to the solution  $\phi_t$ of the equation (\ref{MAF for CRF}) in $C^\infty([T,T']\times Y\setminus \{y_0\})$, then $\pi_*\omega^\e$ smoothly converges to $\omega_t$ on $[T,T']\times Y\setminus \{y_0\}$.

\medskip 
\begin{lem}\label{bound volume}
There exists $C>0$ such that for all $\e\in (0,\e_0)$  such that on $[T,T']\times X$ we have
\begin{enumerate}[label=\textup{(\roman*)}]
\item $\eta \leq  C\omega_0$,

\medskip
\item $Osc_X\psi^\e\leq C$;

\medskip
\item $\frac{(\omega^\e) ^n}{\Omega_X}\leq C.$
\end{enumerate}
\end{lem}

\begin{proof}
By straightforward calculation, in $\pi^{-1}(B(0,1/2))$, we have
\begin{eqnarray*}
 \eta&=&-dd^c\log((\pi^*f_\e)\Omega_X)+dd^c\log((\pi^*f)\Omega_X)\\
 &\leq& (n-1)\pi^*\left(dd^c\log r^2 \right)\\
 &=& (n-1)\frac{\sqrt{-1}}{\pi} \pi^*\left(\frac{1}{r^2} \sum_{i,j}\left(\delta_{ij}-\frac{\bar{z}_i z_j}{r^2} dz_i\wedge d\bar{z}_j\right)\right)\\
 &\leq & C \omega_0,
\end{eqnarray*}
for some constant $C>0$.
This proves (i). Using the same argument in Section \ref{a priori estm} (see Theorem \ref{bound oscillation}) we get (ii). Finally, the estimate (iii)  follows from the same proof for the K\"ahler-Ricci flow (cf. \cite[Lemma 6.2]{SW13})
\end{proof}

\medskip
 Two following lemmas are essential to prove Proposition \ref{two key estimates}.
\begin{lem}\label{trace 1}
There exists $\eta>0$ and a uniform constant $C>0$ such that 
\begin{equation}\label{trace comparison 1}
\omega^\e_t\leq C\frac{\pi^*\omega_Y}{|s|^2_h}.
\end{equation} 
\end{lem}
\begin{proof}
We first denote by $C>0$ a uniform constant  which is independent of $\e$.
Set $\hat{\omega}=\pi^*\omega_Y$ and fix $\delta>0$ a small constant. Following the same method in \cite[Lemma 3.4]{TW13} (see \cite{PS10} for the original idea),  we consider 
$$H_\delta=\log \tr_{\hat{\omega}}\omega^\e_t+\log|s|_h^{2(1+\delta)}-A\psi^\e+\frac{1}{\tilde{\psi}^\e+C_0},$$
where $\tilde{\psi}^\e:=\psi^\e-\frac{1+\delta}{A}\log|s|^2_h$ and $C_0>0$ satisfies $\tilde{\psi}^\e+C_0\geq 1$.  

\medskip
It follows from \cite[(3.17)]{TW13} and \cite[Lemma 2.4]{SW13} that
\begin{equation}\label{omega0 and hat omega}
\omega_0\leq \frac{C}{|s|^2_h}\pi^*\omega_Y,
\end{equation}
 hence $|s|^2_h\tr_{\hat{\omega}}\omega^\e\leq C\tr_{\omega_0}\omega^\e$. Therefore $H_\delta$ goes to negative infinity as $x$ tends to $E$. Suppose that $H_\delta$ attains its maximum at $(t_0, x_0)\in [T,T']\times X\setminus E$. Without loss of generality,  we assume that $\tr_{\hat{\omega}}\omega^\e\geq 1$ at $(t_0,x_0)$.  

\medskip
The condition (\ref{condition of cohomology class}) implies that $\pi^*\omega_Y-\omega_0$ is a $d$-closed form, so that $d\omega_0=\pi^*(d\omega_Y)$. Therefore we have

 \begin{equation}\label{torsion tensor relation}
 (T_0)^{p}_{jl}(g_0)_{p\bar{k}}=(2\sqrt{-1} \partial\omega_0)_{jl\bar{k}}=(2\sqrt{-1}\pi^*\partial  \omega_Y)_{jl\bar{k}}=(\pi^*T_Y)^{p}_{jl}(\pi^*\omega_Y)_{p\bar{k}}.
 \end{equation}
The condition (\ref{condition of cohomology class}) moreover implies that 
\begin{equation}\label{omega epsilon}
\omega^\e(t)=\omega_0+\beta(t),
\end{equation}
where $\beta(t)=(1-\e/T) (-TRic(\omega_0)+dd^c\psi)+(t-T)\pi^*\chi$ is a closed $(1,1)$-form.

\medskip
Combining (\ref{omega epsilon}), (\ref{torsion tensor relation}) and the calculation of \cite[Proposition 3.1]{TW15}, at $(t_0,x_0)$ we get
\begin{eqnarray}\label{evolution of logtr}
\left( \frac{\partial}{\partial t} -\Delta_\e \right)\log\tr_{\hat{\omega}}\omega^\e\leq \frac{2}{(\tr_{\hat\omega}\omega^\e)^2}Re\left(g^{k\bar{q}} \hat{T}^i_{ki}\partial_{\bar{q}}\tr_{\hat{\omega}}\omega^\e  \right) +C\tr_{\omega^\e}\hat{\omega} +\frac{\tr_{\hat{\omega}}\eta}{\tr_{\hat{\omega}}\omega^\e},
\end{eqnarray}
where $\hat{T}:=\pi^*T_Y$. 

\medskip
It follows from Lemma \ref{bound volume}  and (\ref{omega0 and hat omega}) that $$\frac{\tr_{\hat{\omega}}\eta}{\tr_{\hat{\omega}}\omega^\e}\leq \frac{C}{|s|_h^2\tr_{\hat{\omega}}\omega^\e}.$$
Moereover, we may assume without loss of generality that 

$$\frac{C}{|s|_h^2\tr_{\hat{\omega}}\omega^\e}\leq C',$$
for some uniform constant $C'$,
since otherwise $H_\delta$ is already uniformly bounded. Therefore, we get
\begin{equation}
\frac{\tr_{\omega^\e}\eta}{\tr_{\hat{\omega}}\omega^\e}\leq C'.
\end{equation}

\medskip Since  at $(t_0,z_0)$ we have $\nabla H_\delta(t_0,x_0)=0$,
\begin{equation}
\frac{1}{\tr_{\hat{\omega}}\omega^\e}\partial_i\tr_{\hat{\omega}}\omega^\e-A\partial_i\tilde{\psi}^\e-\frac{1}{(\tilde{\psi}^\e+C_0)^2}\partial_i\tilde{\psi}^\e=0,
\end{equation}
hence
\begin{eqnarray*}
\left|\frac{2}{(\tr_{\hat\omega}\omega^\e)^2}Re\left(g^{k\bar{q}} \hat{T}^i_{ki}\partial_{\bar{q}}\tr_{\hat{\omega}}\omega^\e  \right)\right| &\leq& \left|\frac{2}{\tr_{\hat{\omega}}\omega^\e}Re\left(\left(A+\frac{1}{\tilde{\psi}^\e+C_0}\right)g^{k\bar{q}} \hat{T}^i_{ki}(\partial_{\bar{q}} \tilde{\psi}^\e)\right)\right|\\
&\leq& \frac{|\partial \tilde{\psi}^\e|^2_{\omega^\e}}{(\psi^\e+C_0)^3} +CA^2(\tilde{\psi}^\e+C_0)^3\frac{\tr_{\omega^\e}\hat{\omega}}{(\tr_{\hat{\omega}}\omega^\e)^2}.
\end{eqnarray*}
We also have
\begin{eqnarray*}
\left( \frac{\partial}{\partial t} -\Delta_\e \right)\left( -A\tilde{\psi}^\e +\frac{1}{\tilde{\psi}^\e+C_0}\right) =-A\dot{\tilde{\psi}}^\e+A\Delta_\e\tilde{\psi}^\e-\frac{\dot{\tilde{\psi}}^\e}{(\psi^\e+C_0)^2}+\frac{\Delta_\e\tilde{\psi}^\e}{(\tilde{\psi}^\e+C_0)^2}.
\end{eqnarray*}
Combining all inequalities above and Lemma \ref{bound volume} (iii), at $(t_0,x_0)$, we obtain
\begin{eqnarray*}
0\leq \left( \frac{\partial}{\partial t} -\Delta_\e \right) H_\delta&\leq&-\frac{|\partial \tilde{\psi}^\e|^2_{\omega^\e}}{(\psi^\e+C_0)^3} +C\tr_{\omega^\e}\hat{\omega} -\left(A+\frac{1}{(\tilde{\psi}^\e+C_0)^2}  \right)\dot{\psi}^\e\\
& & +C'
+ \left(A+\frac{1}{(\tilde{\psi}^\e+C_0)^2}  \right)\tr_{\omega^\e}(\omega^\e-\theta^\e+\frac{1+\delta}{A}R_h)\\
&\leq& C\tr_{\omega^\e}\hat{\omega}+(A+1)\log\frac{\Omega_X}{\omega^n_\e} +(A+1)n\\
&& -A\tr_{\omega^\e}\left( \theta^\e-\frac{1+\delta}{A}R_h \right)+C.
\end{eqnarray*}

Since  $\pi^*\omega_Y-\e_0R_h>0$ ,  we have
$$\theta^\e-\frac{1+\delta}{A}R_h\geq  c_0\omega_0 $$
for $A$ sufficiently large. Combining with $\tr_{\omega^\e}\hat{\omega}\leq C\tr_{\omega^\e}\omega_0$, we can choose $A$ sufficiently large so that  at $(t_0,x_0)$
$$0\leq -\tr_{\omega^\e}\omega_0+C\log\frac{\Omega_X}{(\omega^\e)^n} +C.$$
Therefore, at $(t_0,x_0)$ 
$$\tr_{\omega_0}\omega^\e \leq \frac{1}{n}(\tr_{\omega^\e}\omega_0)^{n-1}\frac{(\omega^\e)^n}{\omega_0^n} \leq C\frac{\omega^\e}{\Omega_X}\left(\log\frac{\Omega_X}{(\omega^\e)^n}\right)^{n-1}.$$ 

Since $\omega^\e/\Omega_X\leq C$ (Lemma \ref{bound volume} (iii)) and $x\rightarrow x\log|x|^{n-1}$ is bounded from above for $x$ close to zero, we get 
$$\tr_{\omega_0}\omega^\e\leq C,$$
This implies that $H_\delta$ is uniformly bounded from above at its maximum.  Hence we obtain the estimate (\ref{trace comparison 1}).
\end{proof}

\begin{lem}\label{trace 2}
There exists a uniform $\lambda>0$ and $C>0$ such that
$$\omega^\e_t\leq \frac{C}{|s|^{2(1-\lambda)}}\omega_0.$$
\end{lem}
\begin{proof}
Following the method in \cite[Lemma 3.5]{TW13} (see also \cite{PS10}), we consider for each $\delta>0$,
$$H_\delta=\log\tr_{\omega_0}\omega^\e-A \tilde{\f}^\e+\frac{1}{\tilde{\f}^\e+\tilde{C}}+\frac{1}{(\tilde{\psi}^\e+\tilde{C})},$$
where $\tilde{\f}^\e:=-\log(\tr_{\hat{\omega}}\omega^\e |s|^{2(1+\delta)})+A\tilde{\psi}^\e$ and $\tilde{C}$ is chosen so that $\tilde{\f}^\e+\tilde{C}>1$ and $\tilde{\psi}^\e+\tilde{C}>1$. The constant $A>0$ will be chosen hereafter.  Lemma \ref{trace 1} and Lemma \ref{bound volume} (ii) imply that $H_\delta$  goes to negative infinity as $x$ tends to $E$. Hence we can assume that $H_\delta$ attains its maximum at $(t_0,x_0)\in [T,T']\times X\setminus E$. Without loss of generality,  let's assume further that $\tr_{\omega_0}\omega^\e\geq 1$ at $(t_0,x_0)$.

\medskip
As in Lemma \ref{trace 1} we have
\begin{eqnarray*}
\left( \frac{\partial}{\partial t} -\Delta_\e \right)\log\tr_{\hat{\omega}}\omega^\e\leq \frac{2}{(\tr_{\hat\omega}\omega^\e)^2}Re\left(g^{k\bar{q}} \hat{T}^i_{ki}\partial_{\bar{q}}\tr_{\hat{\omega}}\omega^\e  \right) +C\tr_{\omega^\e}\hat{\omega} +\frac{\tr_{\hat{\omega}}\eta}{\tr_{\hat{\omega}}\omega^\e}.
\end{eqnarray*}
and
\begin{eqnarray*}
\left( \frac{\partial}{\partial t} -\Delta_\e \right)\log\tr_{\omega_0}\omega^\e\leq \frac{2}{(\tr_{\omega_0}\omega^\e)^2}Re\left(g^{k\bar{q}} (T_0)^i_{ki}\partial_{\bar{q}}\tr_{\omega_0}\omega^\e  \right) +C\tr_{\omega^\e}\omega_0 +\frac{\tr_{\omega_0}\eta}{\tr_{\omega_0}\omega^\e}, 
\end{eqnarray*}

It follows from Lemma \ref{bound volume} and (\ref{omega0 and hat omega}) that 
$$\frac{\tr_{\hat{\omega}}\eta}{\tr_{\hat{\omega}}\omega^\e}\leq \frac{C}{|s|_h^2\tr_{\hat{\omega}}\omega^\e},$$
and
$$\frac{\tr_{\omega_0}\eta}{\tr_{\omega_0}\omega^\e}\leq \frac{C}{\tr_{\omega_0}\omega^\e}\leq \frac{C}{|s|^2_h\tr_{\hat{\omega}}\omega^\e}.$$

Therefore

\begin{eqnarray*}
\left( \frac{\partial}{\partial t} -\Delta_\e \right)H_\delta  &\leq& C_0 \tr_{\omega_0}\omega^\e +\frac{2}{(\tr_{\omega_0}\omega^\e)^2}Re\left( g^{k\bar{q} }(T_0)^i_{ki} \partial_{\bar{q}}\tr_{\omega^\e}\hat{\omega}\right)+C_0(A+1)\tr_{\omega^\e}\hat{\omega}\\
&&+\left(A+\frac{1}{(\tilde{\f}^\e+\tilde{C})^2} \right) \frac{2}{(\tr_{\hat{\omega}}\omega^\e)^2}Re\left( g^{k\bar{q} }\hat{T}^i_{ki} \partial_{\bar{q}}\tr_{\omega^\e}\hat{\omega}\right)\\
&&-\left( A\left( A+\frac{1}{(\tilde{\f}^\e+\tilde{C})^2}\right)+\frac{1}{(\tilde{\psi}^\e+\tilde{C})^2}   \right)\dot{\psi}^\e\\
&&+\left( A\left( A+\frac{1}{(\tilde{\f}^\e+\tilde{C})^2}\right)+\frac{1}{(\tilde{\psi}^\e+\tilde{C})^2}   \right)\tr_{\omega^\e}\left( \omega^\e-\theta^\e+\frac{(1+\delta)R_h}{A} \right)\\
&&- \frac{2}{(\tilde{\f}^\e+\tilde{C})^3} |\partial \tilde{\f}|^2_g -\frac{2}{(\tilde{\psi}^\e+\tilde{C})^3}|\partial \tilde{\psi}^\e|^2_g+ \frac{CA}{|s|^2_h\tr_{\hat\omega}\omega^\e}.
\end{eqnarray*}
For the last term, we may assume without of generality that 
\begin{eqnarray*}
 \frac{CA}{|s|^2_h\tr_{\hat\omega}\omega^\e}\leq (\tr_{\omega_0}\omega^\e)^{1/A},
\end{eqnarray*}
since otherwise $H_\delta$ is already uniformly  bounded.  Using
$$\tr_{\omega_0}\omega^\e\leq n \left( \frac{(\omega^\e)^n  }{\omega_0^n}\right)(\tr_{\omega^\e }\omega_0)^{n-1},$$
 and $\tr_{\omega_0}\omega^\e\geq 1$ at $(t_0, x_0)$, we get
 \begin{eqnarray}\label{inequality 1}
 \frac{CA}{|s|^2_h\tr_{\hat\omega}\omega^\e}\leq C_1 \tr_{\omega_0}\omega^\e,
\end{eqnarray}
 for $A>n-1$.

It follows from (\ref{curvature of h}) that 
$$\frac{1}{2}A\theta^\e-(1+\delta)R_h\geq c_0\omega_0,$$ for all $A$ sufficiently large. Therefore we can choose $A$ sufficiently large such that 

\begin{eqnarray}\nonumber
A^2\theta^\e-A(1+\delta )R_h&=&\frac{1}{2}A^2\theta^\e+A\left(\frac{1}{2}A\theta^\e-(1+\delta )R_h\right)\\  
&\geq & C_0(A+1)\hat{\omega}+(C_0+C_1+1)\omega_0. \label{inequality 2}
\end{eqnarray}

\medskip
Compute at $(t_0,x_0)$, using (\ref{inequality 1}), (\ref{inequality 2}),  $\tilde{\psi}^\e+\tilde{C} \geq 1$ and $\tilde{\f}^\e +\tilde{C}\geq 1$, 
\begin{eqnarray*}
0&\leq& -\tr_{\omega_0}\omega^\e +\frac{2}{(\tr_{\omega_0}\omega^\e)^2}Re\left( g^{k\bar{q} }(T_0)^i_{ki} \partial_{\bar{q}}\tr_{\omega^\e}\hat{\omega}\right)\\
&&+\left(A+\frac{1}{(\tilde{\f}^\e+\tilde{C})^2} \right) \frac{2}{(\tr_{\hat{\omega}}\omega^\e)^2}Re\left( g^{k\bar{q} }\hat{T}^i_{ki} \partial_{\bar{q}}\tr_{\omega^\e}\hat{\omega}\right)\\
&&-B\dot{\psi}^\e - \frac{2}{(\tilde{\f}^\e+\tilde{C})^3} |\partial \tilde{\f}|^2_g -\frac{2}{(\tilde{\psi}^\e+\tilde{C})^3}|\partial \tilde{\psi}^\e|^2_g +C',
\end{eqnarray*}
for $B$ is a constant in $[A^2, A^2+A+1]$.

\medskip
By the same the argument in \cite[Lemma 3.5]{TW13},  we get, at $(t_0,x_0)$,
\begin{eqnarray}
0\leq -\frac{1}{4}\tr_{\omega^\e}\omega_0-B\log \frac{(\omega^\e)^n}{\Omega_X}+C'.
\end{eqnarray}
As in the proof of Lemma \ref{trace 1}, we infer that $\tr_{\omega_0}\omega^\e$ is bounded from above at $(t_0,x_0)$. Therefore, it follows from Lemma \ref{bound volume} and Lemma \ref{trace 1} that $H_\delta$ is bounded from above uniformly in $\delta$. Let $\delta\rightarrow 0$, we get
$$\log\tr_{\omega_0}\omega^\e+A\log(|s|^2_h\tr_{\hat{\omega}}\omega^\e)\leq C.$$

Since $\tr_{\omega_0}\omega^\e\leq C\tr_{\hat{\omega}}\omega^\e$, we have
$$\log(\tr_{\omega_0}\omega^\e)^{A+1}|s|^{2A}_h\leq C,$$
 and the desired inequality follows with $\lambda=1/(A+1)>0$.
 \end{proof}

\begin{proof}[Proof of Proposition \ref{two key estimates}]

On $Y$, the function $\phi^\e:=\pi_*\psi^\e_t$ satisfies 
\begin{equation}
\frac{\partial \phi^\e}{\partial t}=\log\frac{\left(\hat{\omega}_Y+\frac{\e}{T}(\omega_Y-\pi_*\omega_0)+dd^c\phi^\e\right)^n}{\Omega_Y},  \qquad \psi^\e|_{t=T}=\pi_*\f(T-\e).
\end{equation}
Since $\alpha^\e_t=\hat{\omega}_Y+\frac{\e}{T}(\omega_Y-\pi_*\omega_0)$ is uniformly equivalent to $\hat{\omega}_Y$ for all $\e\in[0,\e_0]$ and $t\in [T,T']$, we can follow the same argument as in Section \ref{a priori estm} to obtain the $C^k$-estimates for $\phi^\e_t$ which are independent of $\e$, for all $t\in (T,T']$. By Azela-Ascoli theorem,  after extracting a subsequence, we can assume that $\phi^\e$ converges  to $\tilde{\phi}$, as $\e\rightarrow 0^+$, in $C^\infty([\delta,T']\times Y)$ for all $\delta\in (T,T')$. Moreover $\tilde{\phi_t}$ uniformly converges to $\phi_T$, hence $\tilde{\phi}$  satisfies (\ref{MAF for CRF}).  Thanks to Theorem \ref{qualitative stab}, $\tilde{\phi}$ is equal to the solution $\phi$ of (\ref{MAF for CRF}). Using Lemma \ref{trace 2} and the standard local parabolic theory, we  obtain the $C^\infty$ estimates $\omega^\e$ on compact sets away from $E$. Hence $\phi$ is the smooth solution of \ref{CRF contraction} on $[T,T']\times Y\setminus \{y_0\}$. Finally,  Proposition \ref{two key estimates} follows directly from Lemma \ref{trace 1} and Lemma \ref{trace 2}.
\end{proof}
Finally, we get the following:
\begin{thm}\label{GH convergence}
The solution $\omega_t$ of (\ref{weak CRF}) smoothly converges to $\omega'$, as $t\rightarrow T^+$, in  $C_{loc}^\infty(Y\setminus\{y_0\})$ and
$(Y,\omega_t)$ converges in the Gromov-Hausdorff sense to $(Y,d_T)$ as $t\rightarrow T^+$.
\end{thm}
\begin{proof}
It follows from the proof of Proposition \ref{two key estimates} that $\phi\in C^\infty([T,T']\times Y\setminus \{y_0\})$, hence $\omega_t$ smoothly converges to $\omega|_{t=T}=\omega'$ in $C^\infty_{loc}(Y\setminus \{y_0\})$.

\medskip
Denote by  $d_{\omega_t}$ the metric induced from $\omega_t$ and $S_r$ the $2n-1$ sphere of radius $r$ in $B$ centered at the origin. Then it follows from Lemma \ref{trace 1} and the argument of \cite[Lemma 2.7(i)]{SW13} that:

\medskip
(a) There exists a uniform constant $C$ such that
\begin{equation}
\textup{diam}_{d_{\omega_t}} (S_r)\leq C, \quad \forall t\in (T,T'].
\end{equation}
Following the same argument of \cite[Lemma 2.7 (ii)]{SW13}, we have

\medskip
(b) For any $z\in B(0,\frac{1}{2})\setminus \{0\}$, the length of a radial path $\gamma(s)=sz$ for $s\in (0,1]$ with respect to $\omega_t$ is uniformly bounded from above by $C|z|^{\lambda}$, where $C$ is a uniformly constant and $\lambda$ as in Lemma \ref{trace 2}.

\medskip
Given (a) and (b), the Gromov-Hausdorff convergence  follows exactly as in \cite[Section 3]{SW13}. This completes the proof of Theorem \ref{GH convergence} and Conjecture \ref{conjecture}.
\end{proof}
\section{Twisted Chern-Ricci flow}\label{definition of CRF}
\subsection{Maximal existence time for the twisted Chern-Ricci flow}\label{existence section}
Let $(X,g)$ be a compact Hermitian manifold of complex dimension $n$. We define here the twisted Chern-Ricci flow on $X$ as
 \begin{equation}\label{TCRF 1}
 \frac{\partial\omega_t }{\partial t}= -Ric(\omega_t)+\eta, \quad \omega|_{t=0}=\omega_0
 \end{equation}
where $ \eta$ is a smooth $(1,1)$-form. 
Set $\hat{\omega}_t=\omega_0+t\eta-tRic(\omega_0)$.
We now define 
\begin{eqnarray*}
T &:=&\sup\{t\geq 0| \exists \psi\in C^\infty (X) \text{ such that } \hat{\omega}_t+dd^c \psi> 0\}\\
  &=&\sup\{T'\geq 0| \forall t\in [0,T'], \exists \psi\in C^\infty (X) \text{ such that } \hat{\omega}_t+dd^c \psi> 0\}.
\end{eqnarray*}

\medskip
We now prove the following theorem generalizing the same result due to Tosatti-Weinkove \cite[Theorem 1.2]{TW15}. We remark that our ingredients for the proof come from a priori estimates proved in Section \ref{a priori estm} which are different from the approach of Tosatti and Weinkove.
\begin{thm}
There exists a unique maximal solution to the twisted Chern-Ricci flow on $[0,T)$. 
\end{thm}
\begin{proof}

Fix $T'<T$. We show that there exists a solution of (\ref{TCRF 1}) on $[0,T']$. First we prove that the twisted Chern-Ricci flow is equivalent to a Monge-Amp\`ere flow. Indeed, consider the following Monge-Amp\`ere flow
\begin{equation}\label{MAF of CRF}
\frac{\partial \f}{\partial t}=\log\frac{(\hat{\omega}_t+dd^c \f)^n}{\omega_0^n}.
\end{equation}
If $\f$ solves (\ref{MAF of CRF}) on $[0.T']$  then taking $\omega_t:=\hat{\omega}_t+dd^c\f$, we get
$$\frac{\partial}{\partial t} (\omega_t-\hat{\omega}_t)=dd^c\log \frac{\omega_t^n}{\omega_0^n},$$
hence $$\frac{\partial}{\partial t}\omega_t=-Ric(\omega_t)+\eta .$$

\medskip
Conversely, if $\omega_t$ solves (\ref{TCRF 1}) on $[0,T']$, then we get
$$\frac{\partial}{\partial t}(\omega_t-\hat{\omega}_t)=-Ric(\omega_t)+Ric(\omega_0)=dd^c\log\frac{\omega^n_t}{\omega_0^n}.$$ 
Therefore if $\f$ satisfies
$$\frac{\partial}{\partial t}(\omega_t-\hat{\omega}_t-dd^c\f)=0,$$
so $\omega_t=\hat{\omega}_t+dd^c\f$ and $\f$ satisfies (\ref{MAF of CRF}).

\medskip
By the standard parabolic theory \cite{Lie}, there exists a maximal solution of $(\ref{MAF of CRF})$ on some time interval $[0,T_{\max})$ with $0<T_{\max}\leq\infty$. We may assume without loss of generality that $T_{\max}<T'$. We now show that a solution of $(\ref{MAF of CRF})$ exists beyond $T_{\max}$. Indeed, the a priori estimates for more general Monge-Amp\`ere flows in Section \ref{a priori estm} gives us uniform estimates for $\f$ in $[0,T_{\max})$ (see Theorem \ref{full estimates}), so we get a solution on $[0,T_{\max}]$. By the short time existence theory the flow (\ref{MAF of CRF}) can go beyond $T_{\max}$, this gives a contradiction. So the twisted Chern-Ricci flow has a solution in $[0,T)$. Finally, the uniqueness of solution follows from Theorem \ref{uniqueness}.
\end{proof}

\subsection{Twisted Einstein metric on Hermitian manifolds}
We fix a smooth  $(1,1)$-form $\eta$.  A solution of the equation
\begin{equation}\label{twisted equation}
Ric(\omega)=\mu\omega+\eta
\end{equation}
with $\mu=1$ or $-1$, is called a {\sl twisted Einstein metric}. We recall
$$\{\eta\}:=\{ \alpha | \exists f\in C^\infty(X) \text{ with }   \alpha= \eta+dd^c f \},$$ 
the  equivalence class of $\eta$.

\medskip
In the sequel we study the convergence of the normalized twisted Chern-Ricci flow to a twisted Einstein metric $\omega=\chi+dd^c\f\in -(c_1^{BC}-\{\eta\})$ assuming that $c_1^{BC}-\{\eta\}<0$ and $\mu=-1$. Note that if $c_1^{BC}(X)<0$ (resp. $c_1^{BC}(X)>0$) implies that $X$ is a K\"ahler manifold which admits a K\"ahler metric in $-c_1(X)$ (resp. in $c_1(X)$). Therefore  the positivity  of the twisted Bott-Chern class is somehow more natural in our context.

\medskip
Assume the twisted first Bott-Chern class $\alpha:= c_1^{BC}(X)-\{\eta\}$ is negative. We now use a result in elliptic Monge-Amp\`ere equation due to Cherrier \cite{Cher87} to prove the existence of twisted Einstein metric. An alternative proof using the convergence of the twisted Chern-Ricci flow will be given in Theorem \ref{convergence 1}.
\begin{thm}\label{twisted Einstein metric}
There exists a unique twisted Einstein metric in $-\alpha$ satisfying (\ref{twisted equation}):
\begin{equation}\label{twisted E equation for negative class}
Ric(\omega)=-\omega+\eta.
\end{equation}
\end{thm}
\begin{proof}
Let $\chi=\eta-Ric(\Omega)$ be a Hermitian metric in $\alpha$, then any Hermitian metric in $\alpha$ can be written as $\omega=\chi+dd^c\f$ where $\f$ is smooth strictly and $\chi$-psh.  Since 
$$\omega-\eta=\chi+dd^c\f-\eta=-Ric(\Omega)-dd^c \f,$$
we get
$$dd^c\log \frac{\omega^n}{\Omega}=-Ric(\omega)+Ric(\Omega)=dd^c\f.$$
Therefore the equation (\ref{twisted E equation for negative class}) can be written as the following
Monge-Amp\`ere equation
\begin{equation}\label{MAE of Twisted equation}
(\chi+dd^c\f)^n=e^{\f}\Omega
\end{equation}  
It follows from \cite{Cher87} that (\ref{MAE of Twisted equation}) admits an unique smooth $\chi$-psh  solution, therefore there exists an unique twisted Einstein metric in $-(c_1^{BC}(X)-\{\eta\})$.
\end{proof}

\subsection{Convergence of the flow when $c_1^{BC}(X)-\{\eta\}<0$ }
We defined the {\sl normalized twisted Chern-Ricci flow} as follows
\begin{equation}\label{NCRF}
\frac{\partial }{\partial t}\omega_t=-Ric(\omega_t)-\omega_t+\eta,
\end{equation}
We have (\ref{NCRF}) is equivalent to the following Monge-Amp\`ere flow
$$\frac{\partial \f}{\partial t}=\log\frac{({\hat\omega}_t^n+dd^c\f)^n}{\Omega}-\f,$$
where $\hat{\omega}_t=e^{-t}+(1-e^{-t})\left(\eta-Ric(\Omega)\right)$ and $\Omega$ is a fixed smooth volume form on $X$. Since we assume $c_1^{BC}(X)-\{\eta\}$ is negative, the flow (\ref{NCRF}) has a longtime solution. The longtime behavior of (\ref{NCRF}) is as follows
\begin{thm}\label{convergence 1}
Suppose $c_1(X)-\{\eta\}< 0$.  Then the normalized twisted  Chern-Ricci flow starting from any initial Hermitian metric $\omega_0$ smoothly converges, as $t\rightarrow +\infty$, to a twisted Einstein Hermitian metric $\omega_{\infty}=\eta-Ric(\Omega)+dd^c \f_{\infty}$ which satisfies
  $$Ric(\omega_\infty)=\eta-\omega_\infty.$$ 
\end{thm}
\begin{proof}
 We now derive the uniform estimates for the solution $\f$ of the following Monge-Amp\`ere 
$$\frac{\partial \f}{\partial t} =\log\frac{(\hat{\omega}_t+dd^c\f)^n}{\Omega}-\f,\quad \f|_{t=0}=0,$$
where $\hat{\omega}_t:= e^{-t}\omega_0+(1-e^{-t})\chi$, and $\chi=\eta-Ric(\Omega)>0$. 

\medskip
The  $C^0$-estimates for $\f$ and $\dot{\f}$ follow from the same arguments as in \cite{Cao85,TZ06,Tsu88} for K\"ahler-Ricci flow (see \cite{TW15} for  the same estimates for the Chern-Ricci flow). Moreover, since
\begin{eqnarray*}
\bigg(\frac{\partial}{\partial t} -\Delta_{\omega_t}\bigg)(\f+\dot{\f}+nt)=\tr_{\omega_t}\chi.
\end{eqnarray*}
and 
$$\bigg(\frac{\partial}{\partial t} -\Delta_{\omega_t}\bigg)(e^t\dot{\f})=-\tr_{\omega_t}(\omega_0-\chi)$$
therefore $$\bigg(\frac{\partial}{\partial t} -\Delta_{\omega_t}\bigg)((e^t-1)\dot{\f}-\f-nt)=-\tr_{\omega_t}\omega_0<0.$$
The maximum principle follows that $(e^t-1)\dot{\f}-\f-nt\leq C$, hence

\begin{equation}\label{exponent upper bound}
\dot{\f}\leq Cte^{-t} 
\end{equation}

\medskip
For the second order estimate,  we follow the method of Tosatti and Weinkove \cite[Lemma 4.1 (iii)]{TW15} in which they have used a technical trick due to Phong and Sturn \cite{PS10}.
\begin{lem}\label{uniform bound of Laplacian}
There exists uniform constant $C>1$ such that
$$\log\tr_{\hat\omega}(\omega_t)\leq C.$$
\end{lem}
\begin{proof}
Since $\f$ is uniformly bounded, we can choose $C_0$ such that $\f+C_0\geq 1$.
Set $$P=\log\tr_{\hat\omega}\omega_t-A\f+\frac{1}{\f+C_0},$$
where $A>0$ will be chosen hereafter.  The idea of adding the third term in $P$ is due to Phong-Sturn \cite{PS10} and was used in the context of Chern-Ricci flow (cf. \cite{TW15}, \cite{TW13},\cite{TW15b}). 

\medskip
Assume without loss of generality that $\tr_{\hat{\omega}}\omega_t\geq 1$ at a maximum point $(t_0,x_0)$ with $t_0>0$ of $P$.
It follows from the same calculation in Lemma \ref{trace 1} that at $(t_0,x_0)$, we have
\begin{eqnarray*}
\bigg(\frac{\partial}{\partial t}-\Delta_{\omega_t}  \bigg)\log \tr_{\hat{\omega}} \omega_t &\leq& \frac{2}{(\tr_{\hat{\omega}}\omega_t)^2}  Re (\hat{g}^{i\bar{l}}g^{k\bar{q}}\hat{T}_{ki\bar{l}}\partial_{\bar{q}}\tr_{\hat{\omega}}\omega_t)+C\tr_{\omega_t}\hat{\omega} +\frac{\tr_{\hat{\omega}}\eta}{\tr_{\hat{\omega}}\omega_t}\\
&\leq& \frac{2}{(\tr_{\hat{\omega}}\omega_t)^2}  Re (\hat{g}^{i\bar{l}}g^{k\bar{q}}\hat{T}_{ki\bar{l}}\partial_{\bar{q}}\tr_{\hat{\omega}}\omega_t)+C\tr_{\omega_t}\hat{\omega}+ C_1,
\end{eqnarray*}
where $C_1>0$ satisfies $\eta\leq C_1\hat{\omega}$.

\medskip
Now at a maximum point $(t_0,x_0)$ with $t_0>0$ we have $\nabla P=0$, hence
$$\frac{1}{\tr_{\hat{\omega}}\omega_t}\partial_{\bar{i}}\tr_{\hat \omega}\omega_t-A\f_{\bar{i}} -\frac{\f_{\bar{i}}}{(\f+C_0)^2}=0.$$
Therefore
\begin{align*}
\bigg| \frac{2}{(\tr_{\hat{\omega}}\omega_t)^2} &Re (\hat{g}^{i\bar{l}}g^{k\bar{q}}\hat{T}_{ki\bar{l}}\partial_{\bar{q}}\tr_{\hat{\omega}}\omega_t)\bigg|\\
&= \bigg| \frac{2}{(\tr_{\hat{\omega}}\omega_t)^2}  Re \bigg( (A+\frac{1}{(\f+C_0)^2})\hat{g}^{i\bar{l}}g^{k\bar{q}}\hat{T}_{ki\bar{l}}\f_{\bar{q}}\bigg)\bigg| \\
&\leq \frac{CA^2}{(\tr_{\hat\omega}\omega_t)^2}(\f+C_0)^3 g^{k\bar{q}}\hat{g}^{i\bar{l}}\hat{T}_{ki\bar{l}}\hat{g}^{m\bar{j}}\overline{\hat{T}_{qj\bar{m}}}+\frac{|\partial \f|^2_{g}}{(\f+C_0)^3}\\
&\leq \frac{CA^2\tr_{\omega_t} \hat{\omega}}{(\tr_{\hat\omega}\omega_t)^2} (\f+C_0)^3+\frac{|\partial \f|^2_{g}}{(\f+C_0)^3}.
\end{align*}
Moreover, we have
\begin{eqnarray*}
\bigg(\frac{\partial}{\partial t}-\Delta_{\omega_t}  \bigg)(-A\f+\frac{1}{\f+C_0})&= &-A\dot{\f}+A\Delta_{\omega_t}\f-\frac{\dot{\f}}{(\f+C_0)^2}\\
&& +\frac{\Delta_{\omega_t}\f}{(\f+C_0)^2}-\frac{2|\partial\f|^2_{g}}{(\f+C_0)^3}\\
&=& -\bigg(A+\frac{1}{(\f+C_0)^2}\bigg)\dot{\f}- \frac{2|\partial\f|^2_{g}}{(\f+C_0)^3}\\
&&+\bigg(A +\frac{1}{(\f+C_0)^2}\bigg)(n-\tr_{\omega_t}\hat{\omega}).
\end{eqnarray*}

Combining these inequalities,  at $(t_0,z_0)$ we have

\begin{eqnarray*}
0\leq \bigg(\frac{\partial}{\partial t}-\Delta_{\omega_t} \bigg)P&\leq&  \frac{CA^2\tr_{\omega_t} \hat{\omega}}{(\tr_{\hat\omega}\omega_t)^2}+C\tr_{\omega_t}\hat{\omega}-(A+\frac{1}{(\f+C_0)^2})\dot{\f}+C_1\\
&&+\bigg( A+\frac{1}{(\f+C_0)^2}\bigg)(n-\tr_{\omega_t}\hat{\omega}) -\frac{|\partial\f|^2_{g}}{(\f+C_0)^3}\\
&\leq& \frac{CA^2\tr_{\omega_t} \hat{\omega}}{(\tr_{\hat\omega}\omega_t)^2}(\f+C_0)^3-C_2+(C-A)\tr_{\omega_t}\hat{\omega}.
\end{eqnarray*}
We can choose $A$ sufficiently large such that at the maximum of $P$ either 
$\tr_{\hat\omega}\omega_t\leq A^2(\f+C_0)^3,$ then we are done, 
or $\tr_{\hat\omega}\omega_t\geq A^2(\f+C_0)^3, $ and $A\geq 2C$. For the second case, we obtain at the maximum of $P$, there exists a uniform constant $C_3>0$ so that
$$\tr_{\omega_t}\hat{\omega}\leq C_3,$$
Hence combining with the following inequality (see for instance \cite[Lemma 4.1.1]{BG13}) 
$$\tr_{\hat{\omega}}\omega_t\leq n(\tr_{\omega_t}\hat{\omega})^{n-1}\frac{\omega_t^n}{\hat{\omega}^{n}},$$
we have
$$\tr_{\hat{\omega}}\omega_t\leq n(\tr_{\omega_t}\hat{\omega})^{n-1}\frac{\omega_t^n}{\hat{\omega}^{n}}\leq C_4.$$
This implies that $P$ is bounded from above at its maximum, so we complete the proof of the lemma. 
\end{proof}
It follows from Lemma \ref{uniform bound of Laplacian} that $\omega_t$ is uniformly equivalent to $\hat{\omega}$ independent of $t$, hence
$$\bigg(\frac{\partial}{\partial t} -\Delta_{\omega_t}\bigg)(e^t\dot{\f})= -\tr_{\omega_t}(\omega_0)+\tr_{\omega_t}\chi\geq -C,$$
hence $\dot{\f}\geq -C(1+t)e^{-t}$ by the maximum principle. Combining with (\ref{exponent upper bound}), we infer that  $\f $ converges uniformly exponentially fast to a continuous function $\f_{\infty}$. Moreover, by the same argument in Section \ref{a priori estm}, Evans-Krylov and Schauder estimates give us the uniform higher order estimates for $\f$. Therefore $\f_{\infty}$ is smooth and $\f_t$ converges to $\f_\infty$ in $C^\infty$.

\medskip
Finally, we get the limiting metric $\omega_\infty=\eta-Ric(\Omega)+dd^c\f_\infty$ which satisfies the twisted Einstein equation
$$Ric(\omega_\infty)=-\omega_{\infty}+\eta.$$ 
This proves the existence of a twisted Einstein metric in $-c_1^{BC}(X)+\{\eta\}$.
\end{proof}
As an application, we prove the existence of a unique solution of the Monge-Amp\`ere equation on Hermitian manifolds. This result was first proved by Cherrier \cite[Th\'eor\`eme 1, p. 373]{Cher87}. 
\begin{thm}
Let $(X,\omega)$ be a Hermitian manifold, $\Omega$ be a smooth volume form on $X$. Then there exists a unique smooth $\omega$-psh fucntion $\f$ satisfying
$$(\omega+dd^c\f)^n=e^{\f}\Omega.$$
\end{thm}
\begin{proof}
Set $\eta=\omega+Ric(\Omega)$, then we have $c_1^{BC}(X)-\{\eta\}<0$. It follows from Theorem \ref{convergence 1} that the twisted normalized Chern-Ricci flow  \begin{equation*}
\frac{\partial }{\partial t}\omega_t=-Ric(\omega_t)-\omega_t+\eta,
\end{equation*}
admits unique solution which smoothly converges to a twisted Einstein Hermitian metric $\omega_\infty=\eta-Ric(\Omega)+dd^c\f_{\infty}=\omega+dd^c\f_{\infty}$ which satisfies $Ric(\omega_\infty)=-\omega_{\infty}+\eta= Ric(\Omega)-dd^c\f_{\infty}$. Therefore $\f_{\infty}$ is a solution  of  the Monge-Amp\`ere equation
$$(\omega+dd^c\f)^n=e^{\f}\Omega$$
The uniqueness of solution follows from the comparison principle.
\end{proof}


\end{document}